\documentclass[pdfmx, a4paper]{article}

\title{Birational description of moduli spaces of rank two logarithmic connections}
\author{Takafumi Matsumoto\footnote{Department of Mathematics, Graduate School of Science, Kobe University, Kobe, Rokko 657-8501, Japan, E-mail: tmatumt@math.kobe-u.ac.jp}}
\date{}

\usepackage[top=25truemm, bottom=25truemm, left=25truemm, right=25truemm]{geometry}
\usepackage{amsmath,amsthm, mathrsfs}
\usepackage{amssymb}
\usepackage{latexsym}
\usepackage{tikz}
\usepackage{tikz-cd}
\usepackage{here}
\usepackage[all]{xy}
\usepackage{bm}

\theoremstyle{definition}
\newtheorem{definition}{Definition}[section]
\newtheorem{theorem}{Theorem}[section]
\newtheorem{proposition}[theorem]{Proposition}
\newtheorem{lemma}[theorem]{Lemma}
\newtheorem{corollary}[theorem]{Corollary}

\newtheorem{remark}[theorem]{Remark}

\newcommand{\res}{\textrm{res}}
\newcommand{\End}{\textrm{End}}
\newcommand{\rank}{\textrm{rank}\,}
\newcommand{\Bun}{\textrm{Bun}}
\newcommand{\App}{\textrm{App}}

\newcommand{\Ker}{\textrm{Ker}\,}
\newcommand{\tr}{\textrm{tr}}
\newcommand{\id}{\textrm{id}}
\newcommand{\Hom}{\textrm{Hom}}

\newcommand{\pardeg}{\textrm{par\,deg}}

\newcommand{\elm}{\textrm{elm}}
\newcommand{\smhom}{\mathcal{S}\mathcal{H}om}
\newcommand{\mhom}{\mathcal{H}om}

\begin{document}
	
	\maketitle
	\begin{abstract}
		In this paper, we provide an explicit description of the Zariski-open subset of the moduli space of rank 2 parabolic logarithmic connections in the case $g\geq 2$. Our approach is to analyze the underlying parabolic bundles and the apparent singularities of the parabolic connections. We prove that a Zariski-open subset of the product of a projective space and the moduli space of parabolic bundles gives a Darboux coordinate for the moduli space of parabolic connections.
	\end{abstract}
	
	\section{Introduction}
	Let $C$ be an irreducible smooth projective curve over the field of complex numbers $\mathbb{C}$ of genus $g$, and let $\mathbf{t}=\{t_1,\ldots,t_n\}$ be a set of $n$ distinct points on $C$. In this paper, we study the geometry of the moduli space $M^{\boldsymbol{\alpha}}(\boldsymbol{\nu}, (L,\nabla_L))$ of $\boldsymbol{\alpha}$-stable $\boldsymbol{\nu}$-parabolic logarithmic connections over $(C,\mathbf{t})$ with fixed determinant $(L,\nabla_L)$. The moduli space of parabolic connections has the canonical symplectic structure, and providing a Darboux coordinate of such a moduli space is important from the viewpoint of the isomonodromic deformation. There are two main approaches to giving a Darboux coordinate. One is to use the apparent singularities and their dual parameters. Okamoto \cite{Ok} described Hamiltonian systems of the Garnier systems, which are obtained from the isomonodromic deformation of rank 2 connections on $\mathbb{P}^1$, by using the apparent singularities and their dual parameters. Iwasaki \cite{Iw} proved that the moduli space of $SL$-connections on a closed Riemann surface of any genus can be locally written by the apparent singularities and their dual parameters as an analytic space and provided Hamiltonian systems of the equations obtained from the isomonodromic deformation in the case of higher genus, which is a generalization of Okamoto's result. Arinkin-Lysenko \cite{AL}, Oblezin \cite{Ob}, Inaba-Iwasaki-Saito \cite{IIS2} and Komyo-Saito \cite{KS} give an explicit description of the moduli space of parabolic connections on $\mathbb{P}^1$ as an algebraic variety. The other approach is to analyze the apparent singularities and underlying parabolic bundles. Loray-Saito \cite{LS} provided an explicit description of the moduli space in the case of $g=0$ in this way. Specifically, they proved that a Zariski-open subset of the moduli space of parabolic connections on $\mathbb{P}^1$ is isomorphic to a Zariski-open subset of the product of a projective space and the moduli space of parabolic bundles. Fassarella-Loray \cite{FL} and Fassarella-Loray-Muniz \cite{FLM} investigated the geometry of the moduli space in the case of $g=1$. In this paper, we describe the Zariski-open subset of the moduli space $M^{\boldsymbol{\alpha}}(\boldsymbol{\nu}, (L,\nabla_L))$ for a certain parabolic weight $\boldsymbol{\alpha}$  in the case $g\geq 2$ by using the apparent singularities and underlying parabolic bundles, which is a generalization of Loray-Saito's result. 
	
	In order to state the description of the Zariski-open subset of the moduli space precisely, we introduce some notations.  Let $\boldsymbol{\nu}=(\nu^\pm_i)_{1 \leq i \leq n}$ be a collection of complex numbers satisfying $\sum_{i=1}^n(\nu^+_i+\nu^-_i)=-d$. Let $\boldsymbol{\alpha}=\{\alpha^{(i)}_1, \alpha^{(i)}_2\}_{1 \leq i \leq n}$ be a collection of rational numbers such that for all $i=1, \ldots, n$, $0 < \alpha^{(i)}_1 < \alpha^{(i)}_2 <1$. Let $(L,\nabla_L)$ be a pair of a line bundle on $C$ with $\deg L=d$ and a logarithmic connection $\nabla_L$ over $L$ which has the residue data $\res_{t_i}(\nabla_L)=\nu^+_i+\nu^-_i$ for each $i$. 
	Let $M^{\boldsymbol{\alpha}}(\boldsymbol{\nu},(L,\nabla_L))$ be the moduli space of rank 2 $\boldsymbol{\alpha}$-stable $\boldsymbol{\nu}$-parabolic connections over $(C,\mathbf{t})$ whose determinant and trace connection are isomorphic to $(L,\nabla_L)$. Inaba~\cite{In} showed that $M^{\boldsymbol{\alpha}}(\boldsymbol{\nu},(L,\nabla_L))$ is a smooth irreducible variety if
	\begin{equation}\label{assume}
		g=1, n\geq 2\ \text{or}\ g \geq 2, n\geq 1.
	\end{equation}
	By elementary transformations, we can change degree $d$ freely. When $d=2g-1$, by the theory of apparent singularities~\cite{SS}, we can define the rational map
	\[
	\App: M^{\boldsymbol{\alpha}}(\boldsymbol{\nu},(L,\nabla_L)) \cdots \rightarrow \mathbb{P}H^0(C,L \otimes \Omega_C^1(D)).
	\]
	The map which forgets connections induces a rational map
	\[
	\Bun: M^{\boldsymbol{\alpha}}(\boldsymbol{\nu},(L,\nabla_L)) \cdots \rightarrow P^{\boldsymbol{\alpha}}(2,L).
	\]
	Let $V_0$ and $M^{\boldsymbol{\alpha}}(\boldsymbol{\nu},(L,\nabla_L))^0$ be the open subsets of $P^{\boldsymbol{\alpha}}(2,L)$ and $M^{\boldsymbol{\alpha}}(\boldsymbol{\nu},(L,\nabla_L))$, respectively, defined in section 5. From Proposition 5.5, we obtain an open immersion $V_0 \hookrightarrow \mathbb{P}H^1(C,L^{-1}(-D))$. Let $\Sigma \subset \mathbb{P}H^0(C,L\otimes \Omega_C^1(D)) \times \mathbb{P}H^1(C,L^{-1}(-D))$ be the incidence variety. Then the following theorem holds.
	\begin{theorem}(Theorem 5.6 and Proposition 5.11)
		Under the condition (1), assume that $d=2g-1$, $\sum_{i=1}^{n}\nu^-_i \neq 0$ and $\sum_{i=1}^n(\alpha^{(i)}_2-\alpha^{(i)}_1)<1$. Then the map 
		\[
		\App\times\Bun \colon  M^{\boldsymbol{\alpha}}(\boldsymbol{\nu},(L,\nabla_L))^0 \longrightarrow (\mathbb{P}H^0(C,L \otimes \Omega_C^1(D)) \times V_0)\,\setminus\,\Sigma
		\]
		is an isomorphism.  Hence, the rational map 
		\[
		\App\times\Bun  \colon M^{\boldsymbol{\alpha}}(\boldsymbol{\nu},(L,\nabla_L))
		\ \cdots \rightarrow |L \otimes \Omega_C^1(D)| \times P^{\boldsymbol{\alpha}}(2,L)
		\]
		is birational. Moreover, $\App$ and $\Bun$ are Lagrangian fibrations.
	\end{theorem}

	We note that for generic $\boldsymbol{\nu}$, the moduli space $M^{\boldsymbol{\alpha}}(\boldsymbol{\nu},(L,\nabla_L))$ is independent of the choice of a parabolic weight $\boldsymbol{\alpha}$. So for generic $\boldsymbol{\nu}$, we can ignore the conditions in the above theorem for a parabolic weight. For other $\boldsymbol{\nu}$, isomorphism classes of the moduli space as varieties depend on the choice of $\boldsymbol{\alpha}$.  In section 4, we investigate the variations of moduli spaces of rank 2 parabolic connections when parabolic weights are changed. 
	
	The rationality problem is a fundamental problem in algebraic geometry. For the moduli space of vector bundles over a smooth projective curve, Tjurin \cite{Tj}, Newstead \cite{Ne1}\cite{Ne2}, and Narasimhan and Ramanan \cite{NR} showed the rationality for special cases. Boden and Yokogawa \cite{BY} showed that moduli space $P^{\boldsymbol{\alpha}}(r, L)$ of $\boldsymbol{\alpha}$-stable parabolic bundles of rank $r$ with the determinant $L$ is rational in the case of certain flag structures including full flag structures. King and Schofield \cite{KS} proved that moduli space $M(r, L)$ of vector bundles of rank $r$ with the determinant $L$ is rational when $r$ and $\deg L$ are coprime, by using the result of Boden and Yokogawa. Hoffmann~\cite{Ho} showed the rationality of the moduli space of parabolic bundles in more cases. Theorem 1.1 implies the rationality of moduli spaces of rank 2 parabolic logarithmic connections with fixed determinant in the case (\ref{assume}). On the other hand, we can see that the rationality of moduli spaces of parabolic logarithmic connections with fixed determinant for any rank  follows from the rationality of $P^{\boldsymbol{\alpha}}(r, L)$ without using $\App\times \Bun$ (see Remark \ref{rationality}). 
	
	The content of this paper is as follows. Section 2 and section 3, respectively, contain a summary of parabolic bundles and parabolic connections. 
	
	In section 4, we study the variations of moduli spaces of rank 2 parabolic connections when parabolic weights cross walls. For parabolic weights $\boldsymbol{\alpha}$ and $\boldsymbol{\beta}$ which are separated by a single wall, we show that the codimension of the locus of $\boldsymbol{\beta}$-unstable parabolic connections on $M^{\boldsymbol{\alpha}}(\boldsymbol{\nu},(L,\nabla_L))$ is equal to $2g+n-3$, which is positive under the assumption (\ref{assume}), by considering the Harder--Narasimhan filtration for parabolic connections.
	
	In section 5, we study the Zariski-open subset of moduli spaces of rank 2 parabolic connections for certain parabolic weights. Firstly, we provide the distinguished open subset $V_0$ of the moduli space of parabolic bundles. Secondly, we introduce the apparent map. The apparent map was defined in general genus and rank by Saito and Szab\'o~\cite{SS}. Thirdly, we prove the first assertion of Theorem 1.1. This proof is based on the proof of Theorem 4.3 in \cite{LS}. We also give another proof that $\App \times \Bun$ is birational.
	Finally, we show that $\App$ and $\Bun$ are Lagrangian fibrations.

	\section{Parabolic bundles}
	
	In this section, we recall basic definitions and known facts on rank 2 parabolic bundles and their moduli spaces. In the case of general rank and parabolic structures, see \cite{MS}.
	\subsection{Parabolic bundles and $\boldsymbol{\alpha}$-stability}
	Let $C$ be a smooth projective curve of genus $g$ and $\mathbf{t}=\{t_1, \ldots, t_n\}$ be a set of $n$ distinct points on $C$. For any algebraic vector bundle $E$ on $C$, we set $E|_{t_i}=E \otimes (\mathcal{O}_C/\mathcal{O}_C(-t_i))$.
	\begin{definition}
		A quasi-parabolic bundle of rank 2 over $(C,\mathbf{t})$ is a pair $(E,l_*=\{l^{(i)}_*\}_{1 \leq i \leq n})$ consisting of the following data:
		\begin{itemize}
			\setlength{\itemsep}{0cm}
			\item[(1)]  $E$ is a rank 2 algebraic vector bundle on $C$.
			\item[(2)] $l^{(i)}_*$ is a filtration $E|_{t_i} = l^{(i)}_0 \supsetneq  l^{(i)}_{1} \supsetneq l^{(i)}_2 =\{0\}$.
		\end{itemize}
	\end{definition}
	The set of filtrations  $ l_*=\{l^{(i)}_*\}_{1\leq i \leq n}$ is said to be a parabolic structure on the vector bundle $E$.
	\\
	
	A weight $\boldsymbol{\alpha}=\{\alpha^{(i)}_1, \alpha^{(i)}_2\}_{1 \leq i \leq n}$ is a collection of rational numbers such that for all $i=1, \ldots, n$, $0 < \alpha^{(i)}_1 < \alpha^{(i)}_2 <1$.
	\begin{definition}
		Let $(E,l_*)$  be a quasi-parabolic bundle and $F$ be a subbundle of $E$. The parabolic degree of $F$ associated with $\boldsymbol{\alpha}$ is defined by
		\[
		\pardeg_{\boldsymbol{\alpha}}F := \deg F + \sum_{i=1}^{n}\sum_{j=1}^{2} \alpha^{(i)}_j \dim ((F|_{t_i} \cap l^{(i)}_{j-1})/(F|_{t_i} \cap l^{(i)}_{j})).
		\]
	\end{definition}
	\begin{definition}
		A quasi-parabolic bundle $(E,l_*)$ is $\boldsymbol{\alpha}$-semistable (resp. $\boldsymbol{\alpha}$-stable) if for any sub line bundle $F \subsetneq E$, the inequality
		\[
		\pardeg_{\boldsymbol{\alpha}}F \underset{(\text{resp}. \ <)}{\leq}
		\frac{\pardeg_{\boldsymbol{\alpha}}E}{2}
		\]
		holds.
	\end{definition}
	\begin{lemma}
		A quasi-parabolic bundle $(E,l_*)$ is $\boldsymbol{\alpha}$-semistable (resp. $\boldsymbol{\alpha}$-stable) if and only if the inequality 
		\[
		\deg E -2 \deg F + \sum_{F|_{t_i}\neq l^{(i)}_1}(\alpha^{(i)}_2-\alpha^{(i)}_1)- \sum_{F|_{t_i}= l^{(i)}_1}(\alpha^{(i)}_2-\alpha^{(i)}_1) 
		\underset{(\text{resp}. \ >)}{\geq}0
		\]
		holds.
	\end{lemma}
	\begin{proof}
		If $F|_{t_i}\neq l^{(i)}_1$, then we have
		\[
		\dim ((F|_{t_i} \cap l^{(i)}_{j-1})/(F|_{t_i} \cap l^{(i)}_{j}))=\left\{
		\begin{array}{ll}
			1\quad j=1 \\
			0 \quad j=2
		\end{array}	
		\right.
		\]
		by definition. If $F|_{t_i} = l^{(i)}_1$, then we also have
		\[
		\dim ((F|_{t_i} \cap l^{(i)}_{j-1})/(F|_{t_i} \cap l^{(i)}_{j}))=\left\{
		\begin{array}{ll}
			0 \quad j=1 \\
			1 \quad j=2
		\end{array}	
		\right..
		\]
		Therefore, we conclude the equivalence.
	\end{proof}
	\subsection{The moduli space of rank 2 quasi-parabolic bundles}
	Let us fix $(C,\mathbf{t})$. Let $P^{\boldsymbol{\alpha}}_{(C,\mathbf{t})}(d)$ denote the moduli space of  rank 2 $\boldsymbol{\alpha}$-semistable quasi-parabolic bundles over $(C,\mathbf{t})$ of degree $d$.
	\begin{theorem}(Mehta and Seshadri [Theorem 4.1~\cite{MS}]).
		The moduli space $P^{\boldsymbol{\alpha}}_{(C,\mathbf{t})}(d)$ is an irreducible  normal projective variety of dimension $4g+n-3$. Moreover, if $(E,l_*)$ is $\boldsymbol{\alpha}$-stable, then $P^{\boldsymbol{\alpha}}_{(C,\mathbf{t})}(d)$ is smooth at the point corresponding to $(E,l_*)$.
	\end{theorem}
	
	Let $\textrm{Pic}^d C$ be the Picard variety of degree $d$, which is the set of isomorphism classes of line bundles of degree $d$ on $C$. Then we can define the morphism 
	\[
	\det \colon P^{\boldsymbol{\alpha}}_{(C,\mathbf{t})}(d) \longrightarrow \textrm{Pic}^d C
	,\ (E,l_*) \mapsto \det E
	\]
	where $\det E=\bigwedge^2 E$. For each $L \in \textrm{Pic}^d C$, set 
	$ P^{\boldsymbol{\alpha}}_{(C,\mathbf{t})}(L) = \text{det}^{-1}(L)$, i.e., 
	\[
	P^{\boldsymbol{\alpha}}_{(C,\mathbf{t})}(L)=
	\{(E,l_*) \in  P^{\boldsymbol{\alpha}}_{(C,\mathbf{t})}(d) \mid \det E \simeq L\}.
	\]
	
	\section{Parabolic connections}
	
	In this section, we review basic definitions and known facts on rank 2 parabolic connections and their moduli spaces. For higher rank bundles, see sections 2, 3, and 5 in \cite{In}.
	
	Let us fix $(C,\mathbf{t})$ and $D=t_1+ \cdots + t_n$ denote the effective divisor associated with $\mathbf{t}$.
	\subsection{Parabolic connections and $\boldsymbol{\alpha}$-stability}
	\begin{definition}
		Let us fix $\lambda \in \mathbb{C}$. We call a pair  $(E, \nabla)$ a logarithmic $\lambda$-connection over $(C,\mathbf{t})$ when $E$ is an algebraic vector bundle on $C$ and $\nabla \colon E \rightarrow E \otimes \Omega_C^1(D)$ satisfies the $\lambda$-twisted Leibniz rule, i.e.,  for all local sections $a \in \mathcal{O}_C, \sigma \in E$,
		\[
		\nabla(a\sigma)=\sigma \otimes\lambda\cdot da+ a\nabla(\sigma).
		\]
	\end{definition}
	Let $(E,\nabla)$ be a rank 2 logarithmic $\lambda$-connection over $(C,\mathbf{t})$. Then we can define the residue matrix $\res_{t_i}(\nabla) \in \End(E|_{t_i}) \simeq M_r(\mathbb{C})$. The set $\{\nu^+_i,\nu^-_i\}$ of eigenvalues of the residue matrix $\res_{t_i}(\nabla)$  is called local exponents of $\nabla$ at $t_i$.
	
	The following lemma follows by the residue theorem.
	\begin{lemma}(Fuchs relation)
		Let $(E,\nabla)$ be a logarithmic $\lambda$-connection over $(C,\mathbf{t})$ and $\boldsymbol{\nu}=(\nu^\pm_i)_{1 \leq i \leq n}$ be local exponents of $\nabla$.  Then we have 
		\[
		\sum_{i=1}^{n}(\nu^+_i+\nu^-_i) = -\lambda \deg E.
		\]
	\end{lemma}
	\begin{remark}
		Fuchs relation also holds for general rank. Let $(E,\nabla)$ be a rank $r$ logarithmic $\lambda$-connection over $(V,\mathbf{t})$ and $\{\nu^{(i)}_1,\cdots,\nu^{(i)}_r\}$ be the set of eigenvalues of $\res_{t_i}(\nabla)$. Then the relation
		\[
		\sum_{i=1}^{n}\sum_{j=1}^{r}\nu^{(i)}_j = -\lambda \deg E.
		\]
		holds.
	\end{remark}
	For each $n \geq 1$, $d \in \mathbb{Z}$ and $\lambda \in \mathbb{C}$, we set 
	\[\mathcal{N}^{(n)}(d,\lambda):=\left\{ (\nu^\pm_i)_{1 \leq i \leq n} \in \mathbb{C}^{2n} \middle|\,\sum_{i=1}^{n}(\nu^+_i+\nu^-_i)=-\lambda d\right\}.
	\] 
	\begin{definition}
		Let us fix $\boldsymbol{\nu}=(\nu^\pm_i)_{1 \leq i \leq n} \in \mathcal{N}^{(n)}(d,\lambda)$. A $\boldsymbol{\nu}$-parabolic $\lambda$-connection over $(C,\mathbf{t})$ is a collection  $(E,\nabla, l_*=\{l^{(i)}_*\}_{1\leq i \leq n})$ consisting of  the following data:
		\begin{itemize}
			\setlength{\itemsep}{0cm}
			\item[(1)]$(E,\nabla)$ is a logarithmic $\lambda$-connection over $(C,\mathbf{t})$.
			\item[(2)]$l^{(i)}_*$ is a parabolic structure of $E$ such that for any $i,j$, $\dim (l^{(i)}_j/l^{(i)}_{j+1})=1
			$ and for any $i$, $(\textrm{res}_{t_i}(\nabla)-\nu^+_i \textrm{id})(l^{(i)}_1)=\{0\}, \, (\textrm{res}_{t_i}(\nabla)-\nu^-_i\textrm{id})(E|_{t_i}) \subset l^{(i)}_{1}$.
		\end{itemize}
		When $\lambda=1$, a $\boldsymbol{\nu}$-parabolic $\lambda$-connection is called $\boldsymbol{\nu}$-parabolic connection.
	\end{definition}
	\begin{remark}
		When $\lambda=0$, $\boldsymbol{\nu}$-parabolic $\lambda$-connections are nothing but $\boldsymbol{\nu}$-parabolic Higgs bundles.
	\end{remark}
	A weight $\boldsymbol{\alpha}=\{\alpha^{(i)}_1, \alpha^{(i)}_2\}_{1 \leq i \leq n}$ is a collection of rational numbers such that for all $i=1, \ldots, n$, $0 < \alpha^{(i)}_1 < \alpha^{(i)}_2 <1$.
	\begin{definition}
		A $\boldsymbol{\nu}$-parabolic $\lambda$-connection $(E,\nabla,l_*)$ is $\boldsymbol{\alpha}$-stable (resp. $\boldsymbol{\alpha}$-semistable) if for any sub line bundle $F \subsetneq E$ satisfying $\nabla(F) \subset F \otimes \Omega_C^1(D)$, the inequality
		\[
		\pardeg_{\boldsymbol{\alpha}}F \underset{(\text{resp}. \ \leq)}{<}
		\frac{\pardeg_{\boldsymbol{\alpha}}E}{2}
		\]
		holds.
	\end{definition}
	\begin{remark}
		Assume that for any collection $(\epsilon_i)_{1 \leq i \leq n}$ with $\epsilon_i \in \{+,-\}$, $\sum_{i=1}^n\nu^{\epsilon_i}_i \notin \mathbb{Z}$. Then a $\boldsymbol{\nu}$-parabolic connection $(E,\nabla,l_*)$ is irreducible, that is, for any sub line bundle $F \subset E$, $\nabla(F) \nsubseteq F \otimes \Omega_C^1(D)$. Moreover, $(E,\nabla,l_*)$ is $\boldsymbol{\alpha}$-stable for any weight $\boldsymbol{\alpha}$ by irreducibility.
		In fact, if a sub line bundle $F \subset E$ satisfies $\nabla(F) \subset F \otimes \Omega_C^1(D)$, then there exists a collection $(\epsilon_i)_{1 \leq i \leq n}$ with $\epsilon_i \in \{+,-\}$ such that $\sum_{i=1}^n\nu^{\epsilon_i}_{i} \in \mathbb{Z}$ by Remark 3.2. 
	\end{remark}
	\subsection{The moduli space of rank 2 parabolic connections}
	For a fixed $(C,\mathbf{t})$, $\boldsymbol{\nu} \in \mathcal{N}^{(n)}(d):=\mathcal{N}^{(n)}(d,1)$ and $\boldsymbol{\alpha}$, let $M^{\boldsymbol{\alpha}}_{(C,\mathbf{t})}(\boldsymbol{\nu}, d)$ be  the coarse moduli space of rank 2 $\boldsymbol{\alpha}$-stable $\boldsymbol{\nu}$-parabolic connections over $(C,\mathbf{t})$.
	\begin{theorem}(Inaba, Iwasaki and Saito [Theorem 2.1 \cite{IIS1}], [Theorem 5.1 \cite{IIS2}], Inaba [Theorem 2.1 \cite{In}])
		$M^{\boldsymbol{\alpha}}_{(C,\mathbf{t})}(\boldsymbol{\nu}, d)$ is an irreducible smooth quasi-projective variety of dimension $8g+2n-6$ if it is nonempty.
	\end{theorem}
	\begin{remark}
		The above result is also true for higher rank bundles, that is, there is the coarse moduli space of rank $r$ $\boldsymbol{\alpha}$-stable $\boldsymbol{\nu}$-parabolic connections over $(C,\mathbf{t})$ and this moduli space is an irreducible smooth quasi-projective variety of dimension $2r^2(g-1)+nr(r-1)+2$ (cf. \cite{In}).
	\end{remark}
	\begin{remark}
		Assume that for any $(\epsilon_i)_{1 \leq i \leq n}$ with $\epsilon_i \in \{+,-\}$, $\sum_{i=1}^n\nu^{\epsilon_i}_i \notin \mathbb{Z}$. Then the moduli space $M^{\boldsymbol{\alpha}}_{(C,\mathbf{t})}(\boldsymbol{\nu}, d)$ 
		does not depend on the weight $\boldsymbol{\alpha}$ by Remark 3.4.
	\end{remark}
	For $\boldsymbol{\nu}=(\nu^\pm_i)_{1\leq i \leq n} \in \mathcal{N}^{(n)}(d,\lambda)$, we set $\tr(\boldsymbol{\nu})=(\nu^+_i+\nu^-_i)_{1\leq i \leq n}$. Let $(L,\nabla_L)$ be a $\tr(\boldsymbol{\nu})$-parabolic $\lambda$-connection, i.e., $L$ is a line bundle on $C$ and $\nabla_L \colon L \rightarrow L \otimes \Omega_C^1(D)
	$ satisfies the $\lambda$-twisted Leibniz rule and $\res_{t_i}\nabla_L = \nu^+_i+\nu^-_i$ for each $i$.
	
	For a fixed $(C,\mathbf{t})$, $\boldsymbol{\nu} \in \mathcal{N}^{(n)}(d), \boldsymbol{\alpha}$, and $(L, \nabla_L)$,  we set 
	\[
	M^{\boldsymbol{\alpha}}_{(C,\mathbf{t})}(\boldsymbol{\nu}, (L,\nabla_L))
	:=\{(E,\nabla,l_*)\in M^{\boldsymbol{\alpha}}_{(C,\mathbf{t})}(\boldsymbol{\nu}, d) \mid (\det E, \tr \nabla)\simeq (L,\nabla_L)\}.
	\]
	\begin{proposition}(Inaba [Proposition 5.1, 5.2, 5.3~\cite{In}]).
		When $g=0, n \geq 4$ or $g=1,n\geq 2$ or $g\geq 2, n \geq 1$, $M^{\boldsymbol{\alpha}}_{(C,\mathbf{t})}(\boldsymbol{\nu}, (L,\nabla_L))$ is an irreducible smooth quasi-projective variety of dimension $6g+2n-6$ if it is nonempty.
	\end{proposition}
	\begin{remark}
		Inaba \cite{In} shows the above result for general rank. The dimension of the moduli space of rank $r$ parabolic connections with the fixed trace connection is given by $(r-1)(2(r+1)(g-1)+nr)$.
	\end{remark}
	\subsection{Elementary transformations}
	Let $(E,\nabla,l_*)$ be a $\boldsymbol{\nu}$-parabolic connection. We consider the parabolic structure at $t_i$
	\[
	E|_{t_i} = l^{(i)}_0 \supsetneq l^{(i)}_1 \supsetneq l^{(i)}_0=0.
	\]
	For the natural homomorphism $\delta_i \colon E \rightarrow E|_{t_i}/l^{(i)}_1$, let $E'=\Ker \delta_i$. Then $E'$ is a locally free sheaf of rank 2 on $C$ and $\nabla$ induces a logarithmic connection $\nabla' \colon E' \longrightarrow E' \otimes \Omega_C^1(D)$.
	We consider a parabolic structure of $E'$ at $t_i$. The natural homomorphism $\pi \colon E' \rightarrow l^{(i)}_1$ induces two exact sequences
	\[
	0 \longrightarrow   E(-t_i) \overset{\iota}{\longrightarrow }E' \overset{\pi}{\longrightarrow } l^{(i)}_1 \longrightarrow 0,
	\]
	\[
	0 \longrightarrow E|_{t_i}/l^{(i)}_1(-t_i) \longrightarrow E'|_{t_i} \longrightarrow l^{(i)}_1 \longrightarrow 0.
	\]
	Set $l'^{(i)}_1=\iota( E|_{t_i}/l^{(i)}_1(-t_i))$, then $\{l'^{(i)}_j\}_j$ defines a parabolic structure of $E'$ at $t_i$.
	\begin{lemma}
		For $\boldsymbol{\nu} \in \mathcal{N}^{(n)}(d)$ and $\boldsymbol{\alpha}$, we define $\boldsymbol{\nu}'=(\nu'^\pm_m)_{1 \leq m \leq n}$ and $\boldsymbol{\alpha}=\{\alpha'^{(i)}_1, \alpha'^{(i)}_2\}_{1\leq i\leq n}$ as 
		\[
		\nu'^\epsilon_m=\left\{
		\begin{array}{lll}
			\nu^\epsilon_m &\quad  m\neq i \\
			\nu^-_i +1&\quad  m=i,\epsilon=+\\
			\nu^+_i   &\quad m=i, \epsilon=-, \quad
		\end{array}
		\right.
		\alpha'^{(m)}_j=\left\{
		\begin{array}{lll}
			\alpha^{(m)}_j &\quad  m\neq i \\
			\alpha^{(m)}_2-1+\theta&\quad  m=i,j=1\\
			 \alpha^{(m)}_1+\theta &\quad m=i, j=2, 
		\end{array}
		\right.
		\]
		respectively, where $\theta$ is a rational number satisfying $0<\alpha^{(i)}_2-1+\theta<\alpha^{(i)}_1+\theta <1$. Then the collection $(E',\nabla',l'_*)$ is a $\boldsymbol{\nu}'$-parabolic connection over $(C,\mathbf{t})$ and $\deg E'=\deg E-1$. $(E,\nabla,l_*)$ is $\boldsymbol{\alpha}$-stable if and only if $(E',\nabla',l'_*)$ is $\boldsymbol{\alpha}'$-stable. 
	\end{lemma}
	We note that $\boldsymbol{\alpha}'$-stability is independent of the choice of $\theta$ by Lemma 2.1.
	\begin{proof}
		Let $e_1, e_2$ be local generators of $E$ around $t_i$ which satisfies $l^{(i)}_1=\mathbb{C}\bar{e}_1$. Here $\bar{e}_j$ denotes the image of $e_j$ by the natural map $E \rightarrow E|_{t_i}$. By definition, we have $\res_{t_i}(\nabla)(\bar{e}_1)=\nu^+_i \bar{e}_1$ and $\res_{t_i}(\nabla)(\bar{e}_2)\equiv\nu^-_i \bar{e}_2\mod l^{(i)}_1$. 
		
		Let $m_{t_i} \subset \mathcal{O}_{C,t_i}$ denote the maximal ideal and $z_i$ be a generator of $m_{t_i}$. By definition, $E'$ is locally generated by $e_1$ and $ z_ie_2$  at $t_i$. Set $e'_1=e_1$ and $e'_2=z_ie_2$. Then the image of $e'_2$ by the natural morphism $E' \rightarrow E'|_{t_i}$ generates $l'^{(i)}_1$. Moreover, we see that near $t_i$,
		\[
		\nabla'(e'_2)=\nabla(z_ie_2)=z_ie_2 \otimes \frac{dz_i}{z_i} +z_i\nabla(e_2),
		\]
		so we obtain $\res_{t_i}(\nabla')(\bar{e}'_2) = (1+\nu^-_i)\bar{e}'_2$. Therefore, the collection $(E',\nabla',l'_*)$ is a $\boldsymbol{\nu}'$-parabolic connection. 
		
		The second assertion follows from the following exact sequence
		\[
		0 \longrightarrow E' \longrightarrow E \longrightarrow E|_{t_i}/l^{(i)}_1 \longrightarrow 0.
		\]
		Suppose that a nonzero subbundle $F\subset E$ destabilizes $\boldsymbol{\alpha}$-stability, that is, $\nabla(F)\subset F\otimes \Omega_{C}^1(D)$ and the inequality
		\[
		\deg E -2 \deg F + \sum_{F|_{t_m}\neq l^{(m)}_1}(\alpha^{(m)}_2-\alpha^{(m)}_1)- \sum_{F|_{t_m}= l^{(m)}_1}(\alpha^{(m)}_2-\alpha^{(m)}_1) <0
		\]
		holds. Let $F'$ be the kernel of the composite $F\rightarrow E\rightarrow E|_{t_i}/l^{(i)}_1.$ Then we have $\nabla'(F')\subset F'\otimes \Omega_{C}^1(D)$ and 
		\[
		\dim (F'|_{t_i}\cap l'^{(i)}_1)=1-\dim (F|_{t_i}\cap l^{(i)}_1). 
		\]
		So we obtain 
		\begin{align*}
				&\deg E' -2 \deg F' + \sum_{F'|_{t_m}\neq l'^{(m)}_1}(\alpha'^{(m)}_2-\alpha'^{(m)}_1)- \sum_{F'|_{t_m}= l'^{(m)}_1}(\alpha'^{(m)}_2-\alpha'^{(m)}_1) \\
				=&\deg E' -2 \deg F'  + (1-2\dim (F'|_{t_i}\cap l'^{(i)}_1))(\alpha'^{(i)}_2-\alpha'^{(i)}_1)+\sum_{\overset{F'|_{t_m}\neq l'^{(m)}_1}{m\neq i}}(\alpha'^{(m)}_2-\alpha'^{(m)}_1)- \sum_{\overset{F'|_{t_m}= l'^{(m)}_1}{m\neq i}}(\alpha'^{(m)}_2-\alpha'^{(m)}_1) \\
				=&(\deg E-1) -2 (\deg F-\dim F|_{t_i}/(F|_{t_i}\cap l^{(i)}_1)) + (2\dim (F|_{t_i}\cap l^{(i)}_1)-1)(\alpha^{(i)}_1-\alpha^{(i)}_2+1)\\
				&+\sum_{\overset{F|_{t_m}\neq l^{(m)}_1}{m\neq i}}(\alpha^{(m)}_2-\alpha^{(m)}_1)- \sum_{\overset{F|_{t_m}= l^{(m)}_1}{m\neq i}}(\alpha^{(m)}_2-\alpha^{(m)}_1) \\
				=&\deg E -2 \deg F + \sum_{F|_{t_m}\neq l^{(m)}_1}(\alpha^{(m)}_2-\alpha^{(m)}_1)- \sum_{F|_{t_m}= l^{(m)}_1}(\alpha^{(m)}_2-\alpha^{(m)}_1) <0.
		\end{align*}
	Conversely, if a nonzero subbundle $F'\subset E'$ breaks $\boldsymbol{\alpha}'$-stability, then we can see that the saturation of $F'$ is an $\boldsymbol{\alpha}$-destabilizing subbundle of $(E,\nabla,l_*)$. 
	\end{proof}
	\begin{definition}
		For a $\boldsymbol{\nu}$-parabolic connection $(E,\nabla,l_*)$, we call the $\boldsymbol{\nu}'$-parabolic connection $(E',\nabla',l'_*)$ given by the above way the lower transformation of $(E,\nabla,l_*)$ at $t_i$ and we set
		\[
		\textrm{elm}^-_{t_i}(E,\nabla, l_*):=(E',\nabla',l'_*).
		\]
	\end{definition}
	Let $E''=E \otimes \mathcal{O}_C(t_i)$ and $\nabla'' \colon E'' \rightarrow E'' \otimes \Omega_C^1(D)$ be the natural logarithmic connection. $l''_*$ is the parabolic structure induced by $l_*$. By the same calculation as above, it follows that $(E'',\nabla'', l''_*)$ is a $\boldsymbol{\nu}''$-parabolic connection over $(C,\mathbf{t})$ where
	\[
	\nu''^\epsilon_m=\left\{
	\begin{array}{ll}
		\nu^\epsilon_m &\quad m\neq i \\
		\nu^\epsilon_i -1 &\quad  m=i
	\end{array}.
	\right.
	\]
	We set $\mathbf{b}_i(E,\nabla,l_*)=(E'',\nabla'', l''_*)$.
	
	\begin{definition}
		For a $\boldsymbol{\nu}$-parabolic connection $(E,\nabla,l_*)$, we define the upper transformation of $(E,\nabla, l_*)$ at $t_i$ by
		\[
		\elm^+_{t_i}(E,\nabla,l_*)=\elm^-_{t_i}\circ \mathbf{b}_i((E,\nabla,l_*)).
		\]
	\end{definition}
	The following lemma follows by definitions of $\elm^-_{t_i}$ and $\mathbf{b}_i$.
	\begin{lemma}
		Put $(E''',\nabla''',l'''_*)=\elm^+_{t_i}(E,\nabla,l_*)$ . For $\boldsymbol{\nu} \in \mathcal{N}^{(n)}(d)$ and $\boldsymbol{\alpha}$, we set
		\[
		\nu'''^\epsilon_m=\left\{
		\begin{array}{lll}
			\nu^\epsilon_m &\quad m\neq i \\
			\nu^-_i &\quad m=i,\epsilon=+ \\
			\nu^+_i - 1 &\quad m=i, \epsilon=-, \quad
		\end{array}
		\right.
		\alpha'''^{(m)}_j=\left\{
		\begin{array}{lll}
			\alpha^{(m)}_j &\quad  m\neq i \\
			\alpha^{(m)}_2-1+\theta&\quad  m=i,j=1\\
			\alpha^{(m)}_1+\theta &\quad m=i, j=2, 
		\end{array}
		\right.
		\]
		respectively, where $\theta$ is a rational number satisfying $0<\alpha^{(i)}_2-1+\theta<\alpha^{(i)}_1+\theta <1$. 
		Then the collection $(E''',\nabla''',l'''_*)$ is a $\boldsymbol{\nu}'''$-parabolic connection over $(C,\mathbf{t})$ and $\deg E'''=\deg E+1$. $(E,\nabla,l_*)$ is $\boldsymbol{\alpha}$-stable if and only if $(E''',\nabla''',l'''_*)$ is $\boldsymbol{\alpha}'''$-stable. 
	\end{lemma}
	\begin{proposition}
		Elementary transformations $\sigma=\elm^\pm_{t_i}$ give an isomorphism between two moduli spaces of stable parabolic connections
		\[
		\sigma \colon M_{(C,\mathbf{t})}^{\boldsymbol{\alpha}}(\boldsymbol{\nu},r,d) \overset{\sim}{\longrightarrow} M_{(C,\mathbf{t})}^{\sigma^*(\boldsymbol{\alpha})}(\sigma^*(\boldsymbol{\nu}),r,d').
		\]
		Here $\sigma^*(\boldsymbol{\alpha})$, $\sigma^*(\boldsymbol{\nu})$ and $d'$ are a parabolic weight, local exponents and degree given in Lemma 3.10 and Lemma 3.11, respectively. Therefore, when we consider 
		some characteristics of moduli spaces of stable parabolic connections, we can freely change degree.
	\end{proposition}
	
	\section{The relation between parabolic weights and the birational types of moduli spaces}
	In this section, we investigate the variations of moduli spaces of rank 2 parabolic connections when parabolic weights are changed. 
	
	\subsection{Parabolic homomorphisms}
	
	Let $(L,\boldsymbol{\beta})$ be a rank 1 parabolic bundle on $(C,\mathbf{t})$, that is, $L$ is a line bundle with the full flag structure over each $t_i$ and $\boldsymbol{\beta}=(\beta_i)_{i=1}^n$ is a set of real numbers satisfying $0<\beta_i<1$ for each $i$. We define parabolic homomorphisms in the case of rank 1 and rank 2. As for the general case, see Definition 1.5 in \cite{MS}.
	\begin{definition}
		Let $(L,\boldsymbol{\beta})$ and $(M,\boldsymbol{\gamma})$ be rank 1 parabolic bundles, and $(E,l_*,\boldsymbol{\alpha})$ be a rank 2 parabolic bundle. A homomorphism $\rho \colon L \rightarrow M$ is said to be parabolic if $\rho(L|_{t_i})=0$ whenever $\beta_i>\gamma_i$, and strongly parabolic if $\rho(L|_{t_i})=0$ whenever $\beta_i\geq\gamma_i$. A homomorphism $\rho \colon L \rightarrow E$ is said to be parabolic if $\rho(L|_{t_i})\subset l^{(i)}_j$ whenever $\beta_i>\alpha^{(i)}_j$, and strongly parabolic if $\rho(L|_{t_i})\subset l^{(i)}_j$ whenever $\beta_i\geq\alpha^{(i)}_j$. A homomorphism $\rho \colon E \rightarrow M$ is said to be parabolic if $\rho(l^{
			(i)}_{j-1})=0$ whenever $\alpha^{(i)}_j>\gamma_i$, and strongly parabolic if $\rho(l^{
			(i)}_{j-1})=0$ whenever $\alpha^{(i)}_j\geq\gamma_i$. 
	\end{definition}
	Let $\mhom((L,\boldsymbol{\beta}), (M,\boldsymbol{\gamma}))$ and $\smhom((L,\boldsymbol{\beta}), (M,\boldsymbol{\gamma}))$ denote the subsheaves of $\mathcal{H}om(L,M)$ consisting of parabolic and strongly parabolic homomorphisms, respectively. Note that $\mhom((L,\boldsymbol{\beta}), (M,\boldsymbol{\gamma}))$ and $\smhom((L,\boldsymbol{\beta}), (M,\boldsymbol{\gamma}))$ depend on only magnitude relationships of $\beta_i$ and $\gamma_i$. 
	\begin{proposition}
		$\smhom((L,\boldsymbol{\beta}), (M,\boldsymbol{\gamma}))(D)$ is dual to $\mhom((M,\boldsymbol{\gamma}),(L,\boldsymbol{\beta}))$.
	\end{proposition}
	\begin{proof}
		Set $I=\{i \mid \beta_i \geq \gamma_i\}$ and $D_I=\sum_{i \in I} t_i$.
		Then we obtain
		\begin{align*}
			\smhom((L,\boldsymbol{\beta}), (M,\boldsymbol{\gamma}))(D)
			&=\mhom(L,M)(D-D_I) \\
			&\cong(\mhom(M,L)(-(D-D_I)))^\vee \\
			&=\mhom((M,\boldsymbol{\gamma}),(L,\boldsymbol{\beta}))^\vee.
		\end{align*}
	\end{proof}
	\begin{lemma}
		Suppose that $\deg L+\sum_{i=1}^n\beta_i > \deg M+\sum_{i=1}^n\gamma_i$. Then 
		\begin{equation*}
			\dim H^0(\mhom((L,\boldsymbol{\beta}),(M,\boldsymbol{\gamma})))=0
		\end{equation*}
	\end{lemma}
	\begin{proof}
		Set	
		\[
		I=\{i_1, \ldots, i_m\}:=\{i \mid \beta_i > \gamma_i\}.
		\]
		By definition, we have
		\[
		\mhom((L,\boldsymbol{\beta}), (M,\boldsymbol{\gamma}))=\mathcal{H}om(L,M)(-t_{i_1}-\cdots-t_{i_m}).
		\]
		and also obtain
		\[
		\deg L > \deg M + \sum_{i \in I} (\gamma_i-\beta_i)+\sum_{i \notin I} (\gamma_i-\beta_i)>\deg M -m.
		\]
		We therefore have $H^0(\mhom((L,\boldsymbol{\beta}),(M,\boldsymbol{\gamma})))=0$.
	\end{proof}

	\subsection{Extensions of parabolic connections}
	
	Given rank 1 parabolic bundles $(L,\boldsymbol{\beta}), (M,\boldsymbol{\gamma})$ over $(C,\mathbf{t})$, an extension of $(M,\boldsymbol{\gamma})$ by $(L,\boldsymbol{\beta})$ is a short exact sequence 
	\begin{equation}
		0 \longrightarrow (L,\boldsymbol{\beta}) \longrightarrow (E, l_*,\boldsymbol{\alpha}) \longrightarrow (M,\boldsymbol{\gamma}) \longrightarrow 0, 
	\end{equation}
	that is, (2) is a short exact sequence of underlying vector bundles, all homomorphisms are parabolic, and $\{\alpha^{(i)}_1,\alpha^{(i)}_2\}=\{\beta_i\} \sqcup \{\gamma_i\}$ for each $i$. Similarly, given rank 1 parabolic connections with a parabolic weight $(L,\nabla_L,\boldsymbol{\beta}), (M,\nabla_M,\boldsymbol{\gamma})$ over $(C,\mathbf{t})$, an extension of $(M,\nabla_M,\boldsymbol{\gamma})$ by $(L,\nabla_L,\boldsymbol{\beta})$ is a short exact sequence 
	\begin{equation}
		0 \longrightarrow (L,\nabla_L,\boldsymbol{\beta}) \longrightarrow (E,\nabla, l_*,\boldsymbol{\alpha}) \longrightarrow (M,\nabla_M,\boldsymbol{\gamma}) \longrightarrow 0, 
	\end{equation}
	that is, (3) is a short exact sequence of underlying parabolic bundles and the following diagram is commutative:
	\[
	\xymatrix@R=25pt@C=20pt{
		0 \ar[r] & L \ar[r] \ar[d]^{\nabla_L}& E \ar[r] \ar[d]^{\nabla}& M \ar[r] \ar[d]^{\nabla_M}& 0 \\
		0 \ar[r] & L\otimes\Omega_{C}^1(D) \ar[r] & E \otimes \Omega_{C}^1(D)\ar[r] & M \otimes \Omega_{C}^1(D)\ar[r] & 0 .
	}
	\]
	Set $\bm{L}=(L,\nabla_L,\boldsymbol{\beta})$ and $\bm{M}=(M,\nabla_M,\boldsymbol{\gamma})$. We define a complex $\textbf{Hom}(\bm{M},\bm{L})$ by 
	\[
	\mhom((M,\boldsymbol{\gamma}),(L,\boldsymbol{\beta})) \longrightarrow \smhom  ((M,\boldsymbol{\gamma}),(L,\boldsymbol{\beta})) \otimes \Omega_C^1(D),\ \ \ f \longmapsto \nabla_L \circ f - f \circ \nabla_M.
	\]
	As in the case of parabolic Higgs bundles (for example, see 3.4 in \cite{Th}), we obtain the following proposition for extensions of parabolic bundles. 
	\begin{proposition}
		The equivalence classes of extensions of $\bm{M}$ by $\bm{L}$ are classified by the hypercohomology group $\bm{H}^1(\textbf{Hom}(\bm{M},\bm{L}))$.
	\end{proposition}
	
	\subsection{Parabolic connections becoming unstable when a wall is crossed}
	
	We consider the space of parabolic weights $V=\{(\alpha^{(i)}_j) \in \mathbb{R}^{2n}\mid 0<\alpha^{(i)}_1<\alpha^{(i)}_2<1 \ \text{for each}\ i \}$. $V$ is separated by a hyperplane given by the equation
	\begin{equation}
		d-2d'+\sum_{i=1}^{n}(2n'_i-1)(\alpha^{(i)}_2-\alpha^{(i)}_1)=0,
	\end{equation}
	where $d$ and $d'$ are integers and $n'_i$ is 0 or 1 for each $i$. We call the intersection of $V$ and a hyperplane of the form (4) a wall. If $d$ is fixed, then the walls are finite in number. 
	
	\begin{proposition}
		If the data $(d,d',n'_i)$ and $(d,d'',n''_i)$ define the same wall  and $(d'',n''_i)\neq (d',n'_i)$, then $(d'',n''_i)= (d-d',1-n'_i)$.
	\end{proposition}
	
	\begin{proof}
		If two equations 
		\begin{align*}
			&d-2d'+\sum_{i=1}^{n}(2n'_i-1)(\alpha^{(i)}_2-\alpha^{(i)}_1)=0,\\
			&d-2d''+\sum_{i=1}^{n}(2n''_i-1)(\alpha^{(i)}_2-\alpha^{(i)}_1)=0
		\end{align*}
		define the same wall, then there is a real number $\lambda$ such that 
		\[
		d-2d''=\lambda(d-2d'), \ \ 2n''_i-1=\lambda(2n'_i-1).
		\]
		For each $i$, $2n'_i-1$ and $2n''_i-1$ take $1$ or $-1$, so $\lambda$ also takes $1$ or $-1$. If $\lambda=1$, then $d''=d'$ and $n''_i=n'_i$. If $\lambda=-1$, then $d''=d-d'$ and $n''_i=1-n'_i$.
	\end{proof}
	
	Take weights $\boldsymbol{\alpha},\boldsymbol{\beta} \in V$ separated by a wall $W$, that is, 
	\begin{align}
		d-2d'+\sum_{i=1}^{n}(2n'_i-1)(\alpha^{(i)}_2-\alpha^{(i)}_1)&>0,\\
		d-2d'+\sum_{i=1}^{n}(2n'_i-1)(\beta^{(i)}_2-\beta^{(i)}_1)&<0.
	\end{align}
	The latter inequality is equivalent to the inequality
	\begin{equation}
		d-2d''+\sum_{i=1}^{n}(2n''_i-1)(\beta^{(i)}_2-\beta^{(i)}_1)>0,
	\end{equation}
	where $d''=d-d'$ and $n''_i=1-n'_i$. Suppose that there is no wall other than $W$ which separates $\boldsymbol{\alpha}$ and $\boldsymbol{\beta}$.
	
	If a parabolic connection $(E,\nabla, l_*)$ with $\deg E=d$ is $\boldsymbol{\alpha}$-stable but $\boldsymbol{\beta}$-unstable, then there is a line subbundle $L$ of $E$ such that $\nabla(L) \subset L \otimes \Omega_C^1(D)$,
	\[
	\pardeg_{\boldsymbol{\alpha}}L<
	\frac{\pardeg_{\boldsymbol{\alpha}}E}{2}
	\]
	and 
	\[
	\pardeg_{\boldsymbol{\beta}}L>
	\frac{\pardeg_{\boldsymbol{\beta}}E}{2}.
	\]
	From (5), (6) and the assumption of the choice of $\boldsymbol{\alpha}$ and $\boldsymbol{\beta}$, we obtain $\deg L=d'$ and $\dim ((L|_{t_i}\cap l^{(i)}_0)/(L|_{t_i}\cap l^{(i)}_1))=n'_i$. We define $\epsilon(m)$ as 
	\[
	\epsilon(m)=\left\{
	\begin{array}{lll}
		+\quad m=0 \\
		- \quad m=1.
	\end{array}
	\right.
	\]
	Set $\nabla_L=\nabla|_L$ and $\nu'_i=\nu_i^{\epsilon (n'_i)}$, then $(L,\nabla_L)$ is a $\boldsymbol{\nu}'$-parabolic connection and we have 
	\[
	\sum_{i=1}^n \nu'_i=\sum_{i=1}^{n}\res_{t_i}(\nabla_L)=-\deg L.
	\]
	by Fuchs relation. Let $M=E/L$ and let $\nabla_M$ be the connection over $M$ induced by $\nabla$. Set $\nu''_i=\nu_i^{\epsilon (n''_i)}$, $\beta'_i=\beta^{(i)}_{2-n'_i}$ and $\beta''_i=\beta^{(i)}_{2-n''_i}$, then $(M,\nabla_M)$ is a $\boldsymbol{\nu}''$-parabolic connection and a sequence of parabolic connections 
	\begin{equation*}
		0 \longrightarrow (L,\nabla_L,\boldsymbol{\beta}') \longrightarrow (E,\nabla, l_*,\boldsymbol{\beta}) \longrightarrow (M,\nabla_M,\boldsymbol{\beta}'') \longrightarrow 0
	\end{equation*}
	is exact. Conversely, assume that $\sum_{i=1}^n \nu'_i$ is an integer. Let $(L,\nabla_L)$ be a $\boldsymbol{\nu}'$-parabolic connection  with degree $d'=\sum_{i=1}^n \nu'_i$ and $(M,\nabla_M)$ be a $\boldsymbol{\nu}''$-parabolic connection  with degree $d-d'$, then any extension of $(M,\nabla_M)$ by $(L,\nabla_L)$ is a  $\boldsymbol{\beta}$-unstable $\boldsymbol{\nu}$-parabolic connection with $d$. Moreover, if the extension is nonsplit, then it is $\boldsymbol{\alpha}$-stable. In fact, suppose that $(E,\nabla, l_*)$ is $\boldsymbol{\alpha}$-unstable and let $K \subset E$ be an $\boldsymbol{\alpha}$-destabilizing sub line bundle, then we have
	\[
	\pardeg_{\boldsymbol{\alpha}}K>
	\frac{\pardeg_{\boldsymbol{\alpha}}E}{2}.
	\]
	Set $m_i=\dim ((K|_{t_i}\cap l^{(i)}_0)/(K|_{t_i}\cap l^{(i)}_1))$,  $\alpha_{K,i}=\alpha^{(i)}_{2-m_i}$ and $\beta_{K,i}=\beta^{(i)}_{2-m_i}$. If the natural composite
	\[
	(K,\boldsymbol{\beta}_K) \longrightarrow (E, l_*,\boldsymbol{\beta}) \longrightarrow (M,\boldsymbol{\beta}'')
	\]
	is zero, then we get the injection $(K,\boldsymbol{\beta}_K) \rightarrow (L,\boldsymbol{\beta}')$. Since whether a map is parabolic depends on only magnitude relationships of parabolic weights, we obtain
	\[
	\mhom((K,\boldsymbol{\beta}_K), (L,\boldsymbol{\beta}'))
	=\mhom((K,\boldsymbol{\alpha}_K), (L,\boldsymbol{\alpha}'))=0.
	\]
	by Lemma 4.2. It is a contradiction, so the composite is not zero. If $(\deg K,(m_i)_i) \neq (d'',(n''_i)_i)$, then from (7) and the assumption of the choice of $\boldsymbol{\alpha}$ and $\boldsymbol{\beta}$ we have
	\[
	\pardeg_{\boldsymbol{\beta}}K>
	\frac{\pardeg_{\boldsymbol{\beta}}E}{2}>
	\pardeg_{\boldsymbol{\beta}}M
	\]
	and this means $\mhom((K,\boldsymbol{\beta}_K), (M,\boldsymbol{\beta}''))=0$ by Lemma 4.2. It is a contradiction. Since $(\deg K,(m_i)_i) = (d'',(n''_i)_i)$ and $\mhom  ((K,\boldsymbol{\beta}_K),(M,\boldsymbol{\beta}''))\neq 0$,  $(K,\boldsymbol{\beta}_K)$ is isomorphic to $(M,\boldsymbol{\beta}'')$. This implies that the extension is split.
	
	Let $L'\subset E$ be another $\boldsymbol{\beta}$-destabilizing sub line bundle. Then we have $\deg L'=d'$ and $\dim ((L'|_{t_i}\cap l^{(i)}_0)/(L'|_{t_i}\cap l^{(i)}_1))=n'_i$. Since 
	\[
	\pardeg_{\boldsymbol{\beta}}L'>
	\frac{\pardeg_{\boldsymbol{\beta}}E}{2}>
	\pardeg_{\boldsymbol{\beta}}M, 
	\]
	the composite
	\[
	(L',\boldsymbol{\beta}') \longrightarrow (E, l_*,\boldsymbol{\beta}) \longrightarrow (M,\boldsymbol{\beta}'')
	\]
	becomes zero. So $L'\subset L$, which implies $L'=L$ since $\deg L'=\deg L$. Hence a $\boldsymbol{\beta}$-destabilizing sub line bundle of $(E,\nabla, l_*)$ is unique.
	
	Putting together the above argument, we get the following proposition. 
	\begin{proposition}\label{pcextex}
		If there is an $\boldsymbol{\alpha}$-stable but $\boldsymbol{\beta}$-unstable $\boldsymbol{\nu}$-parabolic connection with degree $d$, then $\sum_{i=1}^n \nu'_i$ is an integer. When $\sum_{i=1}^n \nu'_i$ is an integer, a $\boldsymbol{\nu}$-parabolic connection $(E,\nabla, l_*)$ with degree $d$ is $\boldsymbol{\alpha}$-stable but $\boldsymbol{\beta}$-unstable if and only if $(E,\nabla, l_*, \boldsymbol{\beta})$ can be expressed as a nonsplit extension of parabolic connections 
		\[
		0 \longrightarrow (L,\nabla_L,\boldsymbol{\beta}') \longrightarrow (E,\nabla, l_*,\boldsymbol{\beta}) \longrightarrow (M,\nabla_M,\boldsymbol{\beta}'') \longrightarrow 0
		\]
		where $(L,\nabla_L)$ is a $\boldsymbol{\nu}'$-parabolic connection with degree $d'=\sum_{i=1}^n \nu'_i$ and $(M,\nabla_M)$ is a $\boldsymbol{\nu}''$-parabolic connection with degree $d-d'$. Moreover, if $(E,\nabla, l_*)$ is $\boldsymbol{\alpha}$-stable but $\boldsymbol{\beta}$-unstable, then $(E,\nabla, l_*)$ has a unique $\boldsymbol{\beta}$-destabilizing sub line bundle.
	\end{proposition}
	\begin{proposition}
		Suppose that $\sum_{i=1}^n \nu'_i$ is an integer, then the dimension of the locus of $\boldsymbol{\beta}$-unstable parabolic connections on $M^{\boldsymbol{\alpha}}_{(C,\mathbf{t})}(\boldsymbol{\nu},d)$ is equal to $6g-3+n$. Moreover, for a $\tr(\boldsymbol{\nu})$-parabolic connection $(K,\nabla_K)$, the dimension of the locus of $\boldsymbol{\beta}$-unstable parabolic connections on $M^{\boldsymbol{\alpha}}_{(C,\mathbf{t})}(\boldsymbol{\nu},(K,\nabla_K))$ is equal to $4g-3+n$.
	\end{proposition}
	\begin{proof}
		Let $(L,\nabla_L)$ be a $\boldsymbol{\nu}'$-parabolic connection with degree $d'=\sum_{i=1}^n \nu'_i$ and $(M,\nabla_M)$ be a $\boldsymbol{\nu}''$-parabolic connection with degree $d-d'$. Since $\beta'_i \neq \beta''_i$ for each $i$, we obtain 
		\[
		\mhom  ((M,\boldsymbol{\beta}''),(L,\boldsymbol{\beta}'))= \smhom  ((M,\boldsymbol{\beta}''),(L,\boldsymbol{\beta}')).
		\]
		Set
		\begin{align*}
			\mathcal{F}^0&=\mhom  ((M,\boldsymbol{\beta}''),(L,\boldsymbol{\beta}')), \\
			\mathcal{F}^1&=\smhom  ((M,\boldsymbol{\beta}''),(L,\boldsymbol{\beta}')) \otimes \Omega_C^1(D), \\
			\mathcal{F}^\bullet&=\mathbf{Hom}((M,\nabla_M,\boldsymbol{\beta}''),(L,\nabla_L,\boldsymbol{\beta}')), 
		\end{align*}
		then we get the long exact sequence
		\[
		0 \rightarrow \bm{H}^0(\mathcal{F}^\bullet)\rightarrow H^0(\mathcal{F}^0)\rightarrow H^0(\mathcal{F}^1)\rightarrow \bm{H}^1(\mathcal{F}^\bullet)\rightarrow H^1(\mathcal{F}^0)\rightarrow H^1(\mathcal{F}^1)\rightarrow \bm{H}^2(\mathcal{F}^\bullet)\rightarrow 0.
		\]
		By Lemma 4.2 and Serre duality, we have
		\begin{align*}
			H^0(\mathcal{F}^0)&=H^0(\mhom  ((M,\boldsymbol{\alpha}''),(L,\boldsymbol{\alpha}')))=0,\\
			H^1(\mathcal{F}^1) &\cong H^0(\mhom((L,\boldsymbol{\beta}'),(M,\boldsymbol{\beta}'')))^*=0
		\end{align*}
		where  $\beta'_i=\beta^{(i)}_{2-n'_i}$ and $\beta''_i=\beta^{(i)}_{2-n''_i}$.
		Hence we obtain $\bm{H}^0(\mathcal{F}^\bullet)=0$ and $\bm{H}^2(\mathcal{F}^\bullet)=0$.
		By Riemann-Roch theorem we  have
		\[
		\dim \bm{H}^1(\mathcal{F}^\bullet)
		= \dim H^0(\mathcal{F}^1) +\dim H^1(\mathcal{F}^0)
		=\deg \Omega_C^1(D)
		=2g-2+n.
		\]
		The dimensions of both moduli spaces of $\boldsymbol{\nu}'$-parabolic connections and $\boldsymbol{\nu}''$-parabolic connections are $2g$ and, $e, \lambda e \in \bm{H}^1(\mathcal{F}^\bullet)$ $(\lambda \in \mathbb{C}^*)$ give the same $\boldsymbol{\nu}$-parabolic connection. By Proposition \ref{pcextex}  the dimension of the locus of  $\boldsymbol{\beta}$-unstable parabolic connections on $M^{\boldsymbol{\alpha}}_{(C,\mathbf{t})}(\boldsymbol{\nu},d)$ is equal to $2g+2g+(2g-2+n)-1=6g-3+n$. In the fixed determinant case, $(L,\nabla_L),(M,\nabla_M)$ have to satisfy $L \otimes M = K, \nabla_L \otimes \nabla_M= \nabla_K$. So $(M,\nabla_M)$ is determined by $M = L^{-1} \otimes K$ and $\nabla_M= \nabla_L^* \otimes \nabla_K$ where $ \nabla_L^*$ is the dual connection of $\nabla_L$. Therefore the dimension of the locus of  $\boldsymbol{\beta}$-unstable parabolic connections on $M^{\boldsymbol{\alpha}}_{(C,\mathbf{t})}(\boldsymbol{\nu},(K,\nabla_K))$ is equal to $2g+(2g-2+n)-1=4g-3+n$. 
	\end{proof}
	\begin{corollary}
		When $g=0, n \geq 4$ or $g=1,n\geq 2$ or $g\geq 2, n \geq 1$, the birational type of $M^{\boldsymbol{\alpha}}_{(C,\mathbf{t})}(\boldsymbol{\nu},d)$ is independent of  the choice of $\boldsymbol{\alpha}$. Moreover, for a $\tr(\boldsymbol{\nu})$-parabolic connection $(K,\nabla_K)$, the birational type of $M^{\boldsymbol{\alpha}}_{(C,\mathbf{t})}(\boldsymbol{\nu},(K,\nabla_K))$ is also independent of the choice of $\boldsymbol{\alpha}$.
	\end{corollary}
	\begin{proof}
		By Theorem 3.5 and Proposition 4.6, the difference of dimensions between $M^{\boldsymbol{\alpha}}_{(C,\mathbf{t})}(\boldsymbol{\nu},d)$ and the locus of  $\boldsymbol{\beta}$-unstable parabolic connections is equal to $(8g+2n-6)-(6g-3+n)=2g+n-3>0$. So the birational type of $M^{\boldsymbol{\alpha}}_{(C,\mathbf{t})}(\boldsymbol{\nu},d)$ is independent of $\boldsymbol{\alpha}$. In the same way, the fixed determinant case follows.
	\end{proof}
	
	\section{The birational structure of moduli spaces of parabolic connections}
	
	In this section, we describe the birational structure of moduli spaces of rank 2 parabolic connections. 
	
	Throughout this section, we assume that $g \geq 1$ and $d=2g-1$. Let us fix a line bundle $L$ with degree $d$. Then we have $H^1(C, L)=\{0\}$ and by Riemann-Roch theorem, $\dim H^0(C, L)=d+1-g=g$. Let us fix a weight $\boldsymbol{\alpha}=\{\alpha^{(i)}_1, \alpha^{(i)}_2\}_{1 \leq i \leq n}$ and set $w_i=\alpha^{(i)}_2-\alpha^{(i)}_1$. 
	\subsection{The distinguished open subset of the moduli space of parabolic bundles}
	\begin{lemma}\label{lemstabeq}
		Assume that $\sum_{i=1}^{n}w_i <1$. For a quasi-parabolic bundle $(E, l_*)$ of rank 2 and odd degree, the following conditions are equivalent:
		\begin{itemize}
			\setlength{\itemsep}{0cm}
			\item[(i)] $(E, l_*)$ is $\boldsymbol{\alpha}$-semistable. 
			\item[(ii)] $(E, l_*)$ is $\boldsymbol{\alpha}$-stable. 
			\item[(iii)] $E$ is stable.
		\end{itemize}
	\end{lemma}
	\begin{proof}
		If $(E, l_*)$ is $\boldsymbol{\alpha}$-semistable but not $\boldsymbol{\alpha}$-stable, then there is a sub line bundle $F \subset E$ such that
		\[
		\deg E -2 \deg F =\sum_{F|_{t_i}= l^{(i)}_1}w_i- \sum_{F|_{t_i}\neq  l^{(i)}_1}w_i.
		\]
		The left hand side is odd, but
		\begin{equation}\label{ll1}
			\left| \sum_{F|_{t_i}= l^{(i)}_1}w_i- \sum_{F|_{t_i}\neq  l^{(i)}_1}w_i\right|\leq\sum_{i=1}^n w_i<1.
		\end{equation}
		It is a contradiction. So conditions (i) and (ii) are equivalent.
		
		If $(E, l_*)$ is $\boldsymbol{\alpha}$-stable, then for all sub line bundle $F \subset E$, the inequality
		\begin{equation}\label{llstab}
			2 \deg F < \deg E + \sum_{F|_{t_i}\neq l^{(i)}_1}w_i- \sum_{F|_{t_i}= l^{(i)}_1}w_i
		\end{equation}
		holds.
		From (\ref{ll1}), it follows that
		\begin{equation*}
			\deg E -1<\deg E +\sum_{F|_{t_i}\neq l^{(i)}_1}w_i- \sum_{F|_{t_i}= l^{(i)}_1}w_i < \deg E +1,
		\end{equation*}
		and so we have $2\deg F \leq \deg E$ by (\ref{llstab}). Since $\deg E$ is odd, we obtain
		\[
		2\deg F \leq \deg E-1  < \deg E.
		\]
		Hence, $E$ is stable. Conversely, if  $E$ is stable, then we can prove that $(E, l_*)$ is $\boldsymbol{\alpha}$-stable by the above argument.
	\end{proof}
	\begin{lemma}\label{lemstab}
		Suppose that a vector bundle $E$ on $C$ satisfies the following conditions:
		\begin{itemize}
			\item[(i)] $E$ is an extension of $L$ by $\mathcal{O}_C$, that is, $E$ fits into an exact sequence
			\[
			0 \longrightarrow \mathcal{O}_C \longrightarrow E \longrightarrow L \longrightarrow 0.
			\]
			\item[(ii)] $\dim H^0(C,E)=1$.
		\end{itemize}
		Then $E$ is stable.
	\end{lemma}
	\begin{proof}
		If $E$ is not stable, then there exists a sub line bundle $F \subset E$ such that $\deg F \geq g$. Since $\dim H^0(C,F)-\dim H^1(C,F)=\deg F +1-g \geq 1$, we have $\dim H^0(C,F) \geq 1$, hence we have an inclusion $\mathcal{O}_C \hookrightarrow F$.
		By assumption (ii), we have a unique inclusion $\mathcal{O}_C \subset F \subset E$, and this inclusion induces the injection 
		$F/\mathcal{O}_C \hookrightarrow E/\mathcal{O}_C \simeq L$. Since $L$ is torsion free, one concludes that $F/\mathcal{O}_C =0$, that is, $F \simeq \mathcal{O}_C$. This contradicts the fact that $\deg F \geq g \geq 1$. 
	\end{proof}
	\begin{proposition}\label{extequiv}
		For an element $b \in H^1(C, L^{-1})$, let 
		\begin{equation}\label{parabext}
			0 \longrightarrow \mathcal{O}_C \longrightarrow E_b \longrightarrow L \longrightarrow 0
		\end{equation}
		be the exact sequence obtained by the extension of $L$ by $\mathcal{O}_C$ with the extension class $b$. Then $\dim H^0(C,E_b)=1$ if and only if the natural cup-product map
		\[
		\langle \ , b \rangle \colon H^0(C, L)\longrightarrow H^1(C,\mathcal{O}_C)
		\]
		is an isomorphism. Moreover, $\dim H^0(C,E_b)=1$ for a generic element $b\in H^1(C, L^{-1})$.
	\end{proposition}
	\begin{proof}
		Since $H^1(C, L)=\{0\}$, from the exact sequence (\ref{parabext}), we obtain the following exact sequence
		\[
		0 \longrightarrow H^0(C,\mathcal{O}_C) \longrightarrow H^0(C,E_b) \longrightarrow H^0(C, L) \overset{\langle \ ,b\rangle}{\longrightarrow} H^1(C,\mathcal{O}_C) \longrightarrow H^1(C,E_b) \longrightarrow 0
		\]
		Here we note that by definition of the extension with $b$ the connecting homomorphism $\delta \colon H^0(C, L) \rightarrow H^1(C,\mathcal{O}_C)$ is given by $\langle\ ,b \rangle$. Since $\dim H^0(C,E_b)=\dim H^1(C,E_b)+ \deg E_b+2(1-g) = \dim H^1(C,E_b)+1$, the first assertion follows from the above exact sequence.
		
		We show the second assertion. We set 
		\[
		Z:=\{(s,b) \in H^0(C, L) \times H^1(C, L^{-1}) \mid \langle s,b\rangle =0\}.
		\]
		Since $\deg L \otimes \Omega_C^1 = 4g-3 \geq 2g-1$, we have $H^1(C, L\otimes \Omega_C^1)=\{0\}$ and
		\[
		\dim H^1(C, L^{-1}) = \dim H^0(C, L \otimes \Omega_C^1)^* = \deg L \otimes \Omega_C^1 +1-g=3g-2.
		\]Hence, it is sufficient to show that $\dim Z=3g-2$. In fact, if $\dim Z=3g-2$, then for generic $b \in H^1(C,L^{-1})$, we have $\dim q^{-1}(b)=0$ and it means $q^{-1}(b)=\{(0,b)\}$. Here $q \colon Z \rightarrow H^1(C, L^{-1})$ is the projection. 
		
		Let $p \colon Z \rightarrow H^0(C, L)$ be the projection. We show that for any $s \in H^0(C, L)\setminus \{0\}$ , $\dim p^{-1}(s)=2g-2$. 
		
		A section $\sigma \in H^0(C,\Omega_C^1)$ induces the diagram
		\[
		\xymatrix@R=30pt@C=40pt{
			H^0(C, L) \times H^1(C, L^{-1}) \ar[r]^-{\langle\ ,\ \rangle }\ar[d]_{\otimes\sigma \times \id} &H^1(C,\mathcal{O}_C) \ar[d]^{\otimes\sigma} \\
			H^0(C, L \otimes \Omega_C^1) \times H^1(C, L^{-1}) \ar[r]^-{\langle\ ,\ \rangle'} &H^1(C,\Omega_C^1) 
		}
		\]
		where the above and below map are natural cup-products and the left and right map are natural maps induced by $\sigma$. Note that $\langle\ , \ \rangle'$ is nondegenerate. Set $s \in H^0(C, L)\setminus\{0\}$. For $b \in H^1(C, L^{-1})$, $\langle s,b \rangle=0$ if and only if for all $\sigma \in H^0(C,\Omega_C^1)$, $\langle s \otimes\sigma , b \rangle'=\langle s,b \rangle\otimes\sigma=0$.
		Since the set
		\[
		\{ s\otimes \sigma \mid \sigma \in H^0(C,\Omega_C^1) \} \simeq H^0(C,\Omega_C^1)
		\]
		is a $g$ dimensional subspace of $H^0(C, L \otimes \Omega_C^1)$ and by the nondegeneracy of $\langle\ ,\ \rangle'$, the set
		\[
		\{b \in H^1(C, L^{-1}) \mid \langle s,b \rangle=0 \}
		\]
		defines a $2g-2$ dimensional subspace of $H^1(C, L^{-1})$. We therefore obtain $\dim p^{-1}(s)=2g-2$. So we conclude $\dim Z=3g-2$.
	\end{proof}
	\begin{proposition}\label{V0}
		Let $\sum_{i=1}^{n}w_i < 1$. Let $V_0 \subset P^{\boldsymbol{\alpha}}(L)=P^{\boldsymbol{\alpha}}_{(C,\bm{t})}(L)$ be the subset which consists of all elements $(E, l_*) \in P^{\boldsymbol{\alpha}}(L)$ satisfying following conditions:
		\begin{itemize}
			\setlength{\itemsep}{0cm}
			\item[(i)] $E$ is an extension of $L$ by $\mathcal{O}_C$.
			\item[(ii)] $\dim H^0(C,E)=1$.
			\item[(iii)] For any $i$, $\mathcal{O}_C|_{t_i} \neq l^{(i)}_1$. Here $\mathcal{O}_C|_{t_i}$ is identified with the image by an injection 
			$\mathcal{O}_C|_{t_i} \hookrightarrow E|_{t_i}$.
		\end{itemize}
		Then $V_0$ is a nonempty Zariski open subset of $P^{\boldsymbol{\alpha}}(L)$.
	\end{proposition}
	\begin{proof}
		Let $E$ be a vector bundle on $C$ satisfying conditions (i) and (ii). Then we have $\det E \simeq L$ from (i) and $E$ is stable by Lemma \ref{lemstab}.
		Let $M_L$ denote the moduli space of rank 2 stable vector bundles on $C$ with the determinant $L$.
		
		First, we show that the subset of $M_L$ consisting of vector bundles satisfying (i) and (ii) is open. Since rank and degree are coprime, $M_L$ has the universal family
		$\mathcal{E}$. Set
		\[
		V=\{x \in M_L\mid \dim H^0(C,\mathcal{E}|_{C \times x})=1\}, 
		\]
		then $V$ is an open subset of $M_L$ by the upper semicontinuity of dimensions. Let $q \colon C \times V \rightarrow V$ be the natural projection. By Corollary 12.9 in \cite{Ha}, $q_*\mathcal{E}$ is an invertible sheaf on $V$ and for any $x \in V$, $(q_*\mathcal{E})|_x$ is naturally isomorphic to 
		$H^0(C,\mathcal{E}|_{C \times x})$.  Hence $q^*q_*\mathcal{E}$ is an invertible sheaf on $C \times V$ and a natural homomorphism $ \iota \colon q^*q_*\mathcal{E} \rightarrow \mathcal{E}$ is injective. By definition, for any $x \in V$, we have $(q^*q_*\mathcal{E})|_{C \times x} \simeq H^0(C, \mathcal{E}|_{C \times x})\otimes_{\mathbb{C}}\mathcal{O}_{C \times x} \simeq\mathcal{O}_{C}$ and $\iota|_{c \times x} \colon \mathcal{O}_C \simeq (q^*q_*\mathcal{E})|_{C \times x}\rightarrow \mathcal{E}|_{C \times x}$ 
		is not zero. Set
		\[
		Y=\{(c,x) \in C \times V \mid \text{$\iota|_{(c,x)} \colon \mathcal{O}_C|_c \simeq (q^*q_*\mathcal{E})|_{(c \times x)}\rightarrow \mathcal{E}|_{(c,x)}$ is zero.}\}
		\]
		and $V'=V\setminus q(Y)$, then $Y$ is a closed subset of $C \times V$ and $V'$ is an open subset of $V$. 
		If $x \in V'$,
		then we obtain $\mathcal{E}|_{C \times x}/\mathcal{O}_C \simeq L$, that is, $\mathcal{E}|_{C \times x}$ is an extension of $L$ by $\mathcal{O}_C$. Therefore, $V'$ consists of all isomorphism classes of vector bundles satisfying the conditions (i) and (ii), and $V'$ is an open subset of $M_L$. Moreover, $V'$ is not empty by Proposition \ref{extequiv}.
		
		Second, we prove that $V_0$ is open. By Lemma \ref{lemstabeq}, we obtain
		\[
		P^{\boldsymbol{\alpha}}(L)\simeq \mathbb{P}(\mathcal{E}|_{t_1\times M_L})\times_{M_L} \mathbb{P}(\mathcal{E}|_{t_2\times M_L})\times_{M_L} \cdots \times_{M_L} \mathbb{P}(\mathcal{E}|_{t_n\times M_L}).
		\]
		For each $t_i$, by projectivization of $\iota|_{t_i \times V'} \colon (q^*q_*\mathcal{E})|_{t_i \times V'} \rightarrow \mathcal{E}|_{t_i \times V'}$, we obtain a morphism $\hat{l}^{(i)}_1 \colon V'\rightarrow \mathbb{P}(\mathcal{E}|_{t_i \times V'})$ such that for all $x \in V'$, $\hat{l}^{(i)}_1(x)$ is the point associated with the image by the immersion $\mathcal{O}_C \hookrightarrow \mathcal{E}|_{C \times x}$ at $t_i$. Let $\varpi \colon P^{\boldsymbol{\alpha}}(L) \rightarrow M_L$ be the natural forgetful map and $p_i \colon P^{\boldsymbol{\alpha}}(L) \rightarrow \mathbb{P}(\mathcal{E}|_{t_i \times M_L})$ be the natural projection. Set
		\[
		V_0=\varpi^{-1}(V')\setminus\bigcup^n_{i=1}p_i^{-1}(\hat{l}^{(i)}_1(V')).
		\]
		Then $V_0$ is an open subset of $P^{\boldsymbol{\alpha}}(L)$ and $V_0$ is the set of all isomorphism classes of parabolic bundles satisfying (i), (ii), and (iii).
	\end{proof}
	We introduce another expression of $V_0$. For $b\in H^1(C, L^{-1})$, let
	\[
	0 \longrightarrow \mathcal{O}_C \longrightarrow E_b \longrightarrow L \longrightarrow 0
	\]
	be the exact sequence associated with $b$. We set 
	\[
	U:=\{b \in H^1(C, L^{-1}) \mid \dim H^0(C,E_b)=1\}
	\]
	and then $U$ is an open subset and $0 \notin U$ by Proposition \ref{extequiv}.  
	
	The natural homomorphism $\psi \colon H^1(C, L^{-1}(-D)) \rightarrow H^1(C, L^{-1})$ induces the morphism
	\[
	\tilde{\psi} \colon \mathbb{P}H^1(C, L^{-1}(-D)) \setminus\,\mathbb{P}\Ker \psi \longrightarrow\mathbb{P}H^1(C, L^{-1}).
	\]
	Let $\tilde{U} \subset \mathbb{P}H^1(C, L^{-1})$ be the open subset associated with $U$ and $\tilde{V}=\tilde{\psi}^{-1}(\tilde{U})$. 
	
	Suppose that $(E, l_*) \in P^{\boldsymbol{\alpha}}(L)$ satisfies conditions (i), (ii), and (iii) of Proposition \ref{V0}. Let $b \in H^1(C, L^{-1})$ be the element associated with an exact sequence
	\[
	0 \longrightarrow \mathcal{O}_C \longrightarrow E \longrightarrow L \longrightarrow 0
	\]
	and $[b]$ be the point in $\mathbb{P}H^1(C, L^{-1})$ associated with the subspace generated by $b$. By assumption, we have $b \in U$. Let $\{U_i\}_i$ be an open covering of $C$ and $(c_{ij})_{i,j}, c_{ij}=
	c_i/c_j$ be transition functions of $L$ over $\{U_i\}_i$. Let $e^i_1$ be the restriction of a global section $\mathcal{O}_C \hookrightarrow E$ on $U_i$ and $e^i_2$ be a local section  of $E$ on $U_i$ whose image by the natural map $E \rightarrow E|_{t_i}$ generates $l^{(k)}_1$ at each $t_k \in U_i$. For generators $e^i_1$ and $e^i_2$, transition matrices $M_{i,j}$ is denoted by
	\[
	M_{ij}
	=
	\begin{pmatrix}
		1 & b'_{ij} \\
		0 & c_{ij}
	\end{pmatrix}
	\]
	where $b'=(b'_{ij}c_j)_{i,j} \in H^1(C, L^{-1}(-D))$. 
	Then we have $\tilde{\psi}([b'])=[b]$, and so $[b'] \in \tilde{V}$. By using the above argument, we can correspond $[b'] \in \tilde{V}$ to an isomorphism class of a parabolic bundle satisfying all conditions of 
	Proposition \ref{V0}. Thus we conclude $V_0 \simeq \tilde{V}$.
	
	Putting together the above argument, we get the following proposition. 
	\begin{proposition}
		Suppose that $\sum_{i=1}^{n}w_i < 1$. Let $V_0 \subset P^{\boldsymbol{\alpha}}(L)$ be the subset defined in Proposition \ref{V0}. Then there is an open immersion $V_0 \hookrightarrow \mathbb{P}H^1(C, L^{-1}(-D))$. 
	\end{proposition}
	\subsection{Apparent map}
	Let us fix $\boldsymbol{\nu} \in \mathcal{N}^{(n)}(d)$ and a $\tr(\boldsymbol{\nu})$-parabolic connection $\nabla_L$ over $L$.
	Let $V_0$ be the open subset of  $P^{\boldsymbol{\alpha}}(L)$ defined in Proposition \ref{V0}.
	We set
	\begin{align*}
		M^{\boldsymbol{\alpha}}(\boldsymbol{\nu},(L,\nabla_L))
		&:=M_{(C,\mathbf{t})}^{\boldsymbol{\alpha}}(\boldsymbol{\nu},(L,\nabla_L)), \\
		M^{\boldsymbol{\alpha}}(\boldsymbol{\nu},(L,\nabla_L))^0
		&:=\{(E,\nabla, l_*) \in M^{\boldsymbol{\alpha}}(\boldsymbol{\nu},(L,\nabla_L)) \mid (E, l_*) \in V_0\}.
	\end{align*}
	
	For each $(E,\nabla, l_*) \in M^{\boldsymbol{\alpha}}(\boldsymbol{\nu},(L,\nabla_L))^0$, $E$ has the unique sub line bundle which is isomorphic to the trivial line bundle. We define the section $\varphi_\nabla \in H^0(C, L \otimes \Omega_C^1(D))$ by the composite
	\[
	\mathcal{O}_C  \hookrightarrow E \xrightarrow{\nabla} E\otimes \Omega_C^1(D) \rightarrow E/\mathcal{O}_C \otimes \Omega_C^1(D) \simeq L \otimes \Omega_C^1(D).
	\]
	
	Suppose that $\varphi_\nabla=0$, i.e. $\nabla(\mathcal{O}_C) \subset \mathcal{O}_C \otimes \Omega_C^1(D)$. Then we obtain $\sum_{i=1}^{n}\nu^-_i = 0$ by Fuchs relation
	because $\mathcal{O}_C|_{t_i} \cap l^{(i)}_1=\{0\}$ for any $i$. So if $\sum_{i=1}^{n}\nu^-_i \neq 0$, then  $\varphi_\nabla\neq 0$ and 
	we therefore define the morphism
	\[
	\App \colon M^{\boldsymbol{\alpha}}(\boldsymbol{\nu},(L,\nabla_L))^0  \longrightarrow \mathbb{P}H^0(C, L\otimes \Omega_C^1(D)) \simeq |L\otimes \Omega_C^1(D)|.
	\]
	\[
	\hspace{-80pt}(E,\nabla, l_*) \longmapsto [\varphi_\nabla]
	\]
	Here $[\varphi_\nabla]$ is the point in  $\mathbb{P}H^0(C, L\otimes \Omega_C^1(D))$ associated with the subspace of $H^0(C, L\otimes \Omega_C^1(D))$ generated by $\varphi_\nabla$. We can extend this map to the rational map 
	\[
	\App \colon M^{\boldsymbol{\alpha}}(\boldsymbol{\nu},(L,\nabla_L)) \cdots \rightarrow |L\otimes \Omega_C^1(D)|.
	\]
	\subsection{Main theorem}
	Let 
	\[
	\Bun\colon M^{\boldsymbol{\alpha}}(\boldsymbol{\nu},(L,\nabla_L))^0  \longrightarrow V_0
	\]
	be the forgetful map which sends $(E,\nabla, l_*)$ to $(E, l_*)$. We can extend this map to the rational map 
	\[
	\Bun \colon M^{\boldsymbol{\alpha}}(\boldsymbol{\nu},(L,\nabla_L)) \cdots \rightarrow P^{\boldsymbol{\alpha}}(L).
	\]
	Let 
	\[
	\langle\ ,\  \rangle \colon H^0(C, L\otimes \Omega_C^1(D)) \times H^1(C, L^{-1}(-D)) \longrightarrow H^1(C,\Omega_C^1)
	\]
	be the natural cup-product. This cup-product is nondegenerate.
	\begin{theorem}\label{MT}
		Assume that $\sum_{i=1}^{n}\nu^-_i \neq 0$ and $\sum_{i=1}^nw_i < 1$. Let us define the subvariety $\Sigma \subset \mathbb{P}H^0(C, L\otimes \Omega_C^1(D)) \times \mathbb{P}H^1(C, L^{-1}(-D))$ by
		\[
		\Sigma=\{([s],[b])\mid \langle s , b \rangle=0\}.
		\]
		Then the map 
		\[
		\App\times\Bun \colon  M^{\boldsymbol{\alpha}}(\boldsymbol{\nu},(L,\nabla_L))^0 \longrightarrow (\mathbb{P}H^0(C, L \otimes \Omega_C^1(D)) \times V_0)\,\setminus\,\Sigma
		\]
		is an isomorphism. Therefore, the rational map 
		\[
		\App\times\Bun  \colon M^{\boldsymbol{\alpha}}(\boldsymbol{\nu},(L,\nabla_L))
		\ \cdots \rightarrow |L \otimes \Omega_C^1(D)| \times P^{\boldsymbol{\alpha}}(L)
		\]
		is birational. In particular, $M^{\boldsymbol{\alpha}}(\boldsymbol{\nu},(L,\nabla_L))$ is a rational variety.
	\end{theorem}
	Before showing this theorem, we prove the following lemma.
	\begin{lemma}\label{lemMT}
		Let $(E, l_*) \in V_0$ and $b \in H^1(C, L^{-1})$ be an element associated with an extension
		\[
		0 \longrightarrow \mathcal{O}_C \longrightarrow E \longrightarrow L \longrightarrow 0.
		\]
		Then the natural cup-product map
		\[
		\langle\ ,b \rangle' \colon H^0(C,\Omega_C^1) \longrightarrow H^1(C, L^{-1}\otimes \Omega_C^1)
		\]
		is an isomorphism. In particular, for an element $b' \in H^1(C, L^{-1}(-D))$ associated with $(E, l_*)$, the composite of the natural cup-product map and the natural homomorphism
		\[
		H^0(C,\Omega_C^1) \xrightarrow{\langle\ ,b' \rangle''}H^1(C, L^{-1}(-D)\otimes \Omega_C^1) \longrightarrow H^1(C, L^{-1} \otimes \Omega_C^1)
		\]
		is also an isomorphism.
	\end{lemma}
	\begin{proof}
		By Serre duality, we have $H^0(C,\Omega_C^1) \simeq H^1(C,\mathcal{O}_C)^*$ and $\ H^1(C, L^{-1} \otimes \Omega_C^1)\simeq H^0(C, L)^*$. So it suffices to prove 
		that the natural cup-product map
		\[
		\langle\ ,b \rangle''' \colon H^0(C, L) \longrightarrow H^1(C,\mathcal{O}_C)
		\]
		is an isomorphism, and it is nothing but the first assertion of Proposition 5.3.
		
		The second assertion follows from the following diagram.
		\[
		\xymatrix@R=30pt@C=40pt{
			H^0(C,\Omega_C^1) \times H^1(C, L^{-1}(-D)) \ar[r]^-{\langle\ ,\ \rangle'' }\ar[d]
			&H^1(C, L^{-1}(-D)\otimes\Omega_C^1) \ar[d]\\
			H^0(C,\Omega_C^1) \times H^1(C, L^{-1}) \ar[r]^-{\langle\ ,\ \rangle'} &H^1(C, L^{-1}\otimes\Omega_C^1) 
		}
		\]
	\end{proof}
	\begin{proof}
		(Proof of Theorem \ref{MT})
		
	Firstly, we show that for any $\gamma \in H^0(C, L \otimes \Omega_C^1 (D))$ and $b \in H^1(C, L^{-1}(-D))$ such that the quasi-parabolic bundle $(E, l_*)$ associated with $b$ is in $V_0$, there exist a unique complex number $\lambda$ and a unique $\lambda\boldsymbol{\nu}$-parabolic $\lambda$-connection $(E,\nabla, l_*)$ such that $\tr \nabla=\lambda\nabla_L$ and $\varphi_\nabla=\gamma$. 
		
		Let $\{U_i\}_i$ be an open covering of $C$ and $(c_{ij})_{i,j}, c_{ij}=
		c_i/c_j$ be transition functions of $L$ over $\{U_i\}_i$. Let $e^i_1$ be the restriction of a global section $\mathcal{O}_C \hookrightarrow E$ on $U_i$ and $e^i_2$ be a local section  of $E$ on $U_i$ whose image $\bar{e}^i_2$ by the natural map $E \rightarrow E|_{t_i}$ generates $l^{(k)}_1$ at each $t_k \in U_i$. For local generators $e^i_1$ and $e^i_2$, we can denote transition matrices of $E$ by
		\[
		M_{ij}
		=
		\begin{pmatrix}
			1 & b_{ij} \\
			0 & c_{ij}
		\end{pmatrix},
		\]
		where $b=(b_{ij}c_j)_{i,j} \in H^1(C, L^{-1}(-D))$ is the cocycle corresponding to  an extension
		\[
		0\longrightarrow \mathcal{O}_C\longrightarrow E\longrightarrow L\longrightarrow 0.
		\]
		
		A logarithmic $\lambda$-connection $\nabla$ is given in $U_i$ by $\lambda d + A_i$
		\[
		A_i=
		\begin{pmatrix}
			\alpha_i& \beta_i \\
			\gamma_i & \delta_i
		\end{pmatrix}
		\in M_2(\Omega_C^1(D)(U_i))
		\]
		with the compatibility condition
		\[
		\lambda dM_{ij}+A_iM_{ij}=M_{ij}A_j
		\]
		on each intersection $U_i \cap U_j$. By using elements of matrices, this condition is written by
		\begin{equation}\label{4eq}
			\left\{
			\begin{array}{ll}
				\dfrac{\gamma_i}{c_i}- \dfrac{\gamma_j}{c_j}=0 \\
				\alpha_i-\alpha_j=b_{ij}\gamma_j \\
				\delta_i-\delta_j=-b_{ij}\gamma_j-\lambda \dfrac{dc_{ij}}{c_{ij}} \\
				c_i\beta_i-c_j\beta_j=-(\lambda c_jdb_{ij}+(b_{ij}c_j)(\alpha_i-\delta_j)).
			\end{array}
			\right.
		\end{equation}
		If $(E,\nabla, l_*)$ is a $\lambda\boldsymbol{\nu}$-parabolic $\lambda$-connection, then for each point $t_i$, $\nabla$ satisfies the residual condition
		\begin{equation}\label{eigenA}
			\res_{t_k}(A_i)=
			\begin{pmatrix}
				\lambda \nu_k^- & 0 \\
				* & \lambda \nu_k^+
			\end{pmatrix}
		\end{equation}
		at each $t_k \in U_i$ because $\bar{e}^i_2$ generates $l^{(k)}_1$.
		$\nabla_L$ is denoted in $U_i$ by $d+\omega_i$ with the compatibility condition
		\begin{equation}\label{connL}
			dc_{ij}+c_{ij}\omega_i=c_{ij}\omega_j
		\end{equation}
		on each $U_i \cap U_j$. If $\tr \nabla = \lambda\nabla_L$, then the equation
		\begin{equation}\label{trcond}
			\alpha_i+\delta_i=\lambda\omega_i
		\end{equation}
		holds. When $\nabla$ is denoted in $U_i$ by $\lambda d + A_i$, we have $\varphi_\nabla=(\gamma_i/c_i)_i\in H^0(C, L \otimes \Omega_C^1(D))$. So if $\varphi_\nabla=\gamma$, then we have
		\begin{equation}\label{gamma}
			(\gamma_i/c_i)_i=\gamma.
		\end{equation}
		We show that there exist $\lambda \in \mathbb{C}$ and $\alpha_i,\beta_i,\gamma_i,\delta_i \in \Omega_C^1(D)(U_i)$ satisfying the conditions (\ref{4eq}), (\ref{eigenA}), (\ref{trcond}) and (\ref{gamma}) uniquely.
		
		Step 1: we find $\gamma_i$. From (\ref{gamma}), we have to set $\gamma_i=c_i\gamma$.
		
		Step 2: we find $\alpha_i$. Fix a section $\alpha^0_i \in \Omega_C^1(D)(U_i)$ which has the residue data $\res_{t_k}(\alpha^0_i)=\nu^-_k$ at each $t_k \in U_i$. The cocycle $(\alpha^0_i-\alpha^0_j)_{i,j}$ defines an element of $H^1(C,\Omega_C^1)$. If $(\alpha^0_i-\alpha^0_j)_{i,j}$ is zero in $H^1(C,\Omega_C^1)$, then there exist sections 
		$\tilde{\alpha}_i \in \Omega_C^1(U_i)$ on each $i$ such that $\alpha^0_i-\alpha^0_j=\tilde{\alpha}_i-\tilde{\alpha}_j$ for any $i,j$. $(\alpha^0_i-\tilde{\alpha_i})_i$ defines a global logarithmic 1-form whose sum of residues $\sum_{i=1}^n \nu^-_i$ is not zero. This contradicts the residue theorem. Therefore, the cocycle $(\alpha^0_i-\alpha^0_j)_{i,j}$ is a generator of $H^1(C,\Omega_C^1)$ 
		and there is a unique complex number $\lambda$ such that $\lambda(\alpha^0_i-\alpha^0_j)_{i,j}=(b_{ij}\gamma_j)_{i,j}$. Let $\tilde{\alpha}_i \in \Omega_C^1(U_i)$ be a section such that 
		\[
		\tilde{\alpha}_i-\tilde{\alpha}_j=b_{ij}\gamma_j-\lambda(\alpha^0_i-\alpha^0_j)
		\]
		for any $i,j$. Set $\alpha_i=\lambda \alpha^0_i+\tilde{\alpha}_i$, then $(\alpha_i)_i$ is a solution of the second equation of (\ref{4eq}) and has the residue data $\res_{t_k}(\alpha_i)=\lambda\nu^-_k$. Note that $(\alpha_i)_i$ is still not uniquely determined. Actually, the difference of two solutions of 
		the second equation of (\ref{4eq}) having the same residue data defines a global 1-form and now $\dim H^0(C,\Omega_C^1) \geq g \geq 1$.
		
		Step 3: we find $\delta_i$. From (\ref{trcond}), we have to set $\delta_i=\lambda\omega_i-\alpha_i$. It is clear that $(\delta_i)_i$ is a solution of the third equation of (\ref{4eq}) and has the residue data $\res_{t_k}(\delta_i)=\lambda\nu^+_k$. $\delta_i$ is uniquely determined by $\alpha_i$.
		
		Step 4: we find $\beta_i$ and show that $\alpha_i$ is uniquely determined. From the cocycle condition of $(b_{ij}c_j)_{i,j}$ and the first, second, and third equations of (\ref{4eq}), we obtain
		\begin{align*}
			&(\lambda c_jdb_{ij}+(b_{ij}c_j)(\alpha_i-\delta_j))+(\lambda c_kdb_{jk}+(b_{jk}c_k)(\alpha_j-\delta_k))\\
			=&-\lambda b_{ij}c_{k}dc_{jk}+\lambda c_kdb_{ik}+(b_{ik}c_{k}-b_{jk}c_{k})\alpha_i-b_{ij}c_{j}\delta_j+(b_{jk}c_k)(\alpha_j-\delta_k)\\
			=&-\lambda b_{ij}c_{k}dc_{jk}+\lambda c_kdb_{ik}+b_{ik}c_k(\alpha_i-\delta_k)-b_{jk}c_k(\alpha_i-\alpha_j)-b_{ij}c_j(\delta_j-\delta_k)\\
			=&\lambda c_kdb_{ik}+b_{ik}c_k(\alpha_i-\delta_k).
		\end{align*}
		 So $(-(\lambda c_jdb_{ij}+(b_{ij}c_j)(\alpha_i-\delta_j)))_{i,j}$ defines a cocycle of $H^1(C, L^{-1}\otimes \Omega_C^1)$. Note that a solution of the fourth equation of (\ref{4eq}) exists if and only if $(-(\lambda c_jdb_{ij}+(b_{ij}c_j)(\alpha_i-\delta_j)))_{i,j} $ is trivial. We denote the image of $b$ by the natural homomorphism $H^1(C, L^{-1}(-D))\rightarrow H^1(C, L^{-1})$ by the same character $b$. Since the linear map $\langle\ , b\rangle'' \colon H^0(C,\Omega_C^1)\rightarrow H^1(C, L^{-1}\otimes \Omega_C^1)$ is an isomorphism by Lemma \ref{lemMT}, there exists a unique global 1-form $\zeta=(\zeta_i/c_i)_i \in H^0(C,\Omega_C^1)$ such that
		\[
		(2b_{ij}\zeta_j)_{i,j}=\langle2\zeta ,b\rangle''=-(\lambda c_jdb_{ij}+(b_{ij}c_j)(\alpha_i-\delta_j))_{i,j}, 
		\]
		that is, 
		\[
		-(\lambda c_jdb_{ij}+(b_{ij}c_j)((\alpha_i+\zeta_i/c_i)-(\delta_j-\zeta_j/c_j)))_{i,j}=0
		\]
		in $H^1(C, L^{-1}\otimes \Omega_C^1)$. 
		So there exist unique $(\alpha_i)_i$ and $(\delta_i)_i$ satisfying the condition (\ref{trcond}) and 
		\[
		-(\lambda c_jdb_{ij}+(b_{ij}c_j)(\alpha_i-\delta_j))_{i,j}=0,
		\]
		and there exists  a solution of the fourth equation $(\beta_i)_i$ of (\ref{4eq}) such that $\res_{t_k}(\beta_i)=0$ for any $i$ and $t_k \in U_i$.  Since 
		$H^0(C, L^{-1}\otimes \Omega_C^1)\simeq H^1(C, L)^* = \{0\}$, $(\beta_i)_i$ is uniquely determined. 
		
		When $\lambda=0$, the cocycle $(b_{ij}\gamma_j)_{i,j}$ is zero because $\alpha_i \in \Omega_C^1(U_i)$. Conversely, assume that $(b_{ij}\gamma_j)_{i,j}=0$.
		Then there exists $\tilde{\alpha}_i \in \Omega_C^1(U_i)$ for each $i$ such that $\alpha_i-\alpha_j=b_{ij}\gamma_j=\tilde{\alpha}_i-\tilde{\alpha}_j$. The cocycle $(\alpha_i - \tilde{\alpha}_i)_i$ defines a global logarithmic 1-form on $C$. By the residue theorem, we have
		\[
		\sum_{i=1}^n\lambda \nu^-_i=0.
		\]
		By assumption, we obtain $\lambda=0$. 
		
		For a point $([\gamma], [b])\in (\mathbb{P}H^0(C, L \otimes \Omega_C^1(D)) \times V_0)\,\setminus\,\Sigma$, there exist a unique complex number $\lambda$ and a unique $\lambda\boldsymbol{\nu}$-parabolic $\lambda$-connection $(E,\nabla, l_*)$ such that $\tr \nabla=\lambda\nabla_L$, $\varphi_\nabla=\gamma$, and $(E,l_*)$ is the quasi-parabolic bundle corresponding to $b$. Then $\lambda\neq 0$ and $(E, \lambda^{-1}\nabla, l_*)$ is a $\boldsymbol{\nu}$-parabolic connection with the determinant $(L,\nabla_L)$ whose image by $\App\times\Bun$ is $([\gamma], [b])$. If a $\boldsymbol{\nu}$-parabolic connection $(E,\nabla',l_*)$ satisfies $\tr \nabla'=\nabla_L$ and  $\varphi_{\nabla'}\in [\gamma]$, then there is a unique complex number $\mu$ such that $\varphi_{\nabla'}=\mu \lambda^{-1}\gamma$. A $\mu\boldsymbol{\nu}$-parabolic $\mu$-connection $(E,\mu\lambda^{-1}\nabla,l_*)$ satisfies $\tr (\mu\lambda^{-1}\nabla)=\mu\nabla_L$ and  $\varphi_{\mu\lambda^{-1}\nabla}=\mu \lambda^{-1}\gamma$, so we have $\mu=1$ and $\nabla'=\lambda^{-1}\nabla$ by the uniqueness.
		Therefore, the morphism
		\[
		\App\times\Bun \colon  M^{\boldsymbol{\alpha}}(\boldsymbol{\nu},(L,\nabla_L))^0 \longrightarrow (\mathbb{P}H^0(C, L \otimes \Omega_C^1(D)) \times V_0)\,\setminus\,\Sigma
		\]
		is bijective. By Zariski's main theorem (for example, see Chapter 3, \S 9, Proposition 1 in \cite{Mu}), $\App \times \Bun$ is an isomorphism.
	\end{proof}
	The following proposition is the same as Proposition 4.6 in \cite{LS} and follows by using the same argument of the proof.
	\begin{proposition}
		Suppose that $\sum_{i=1}^n\nu^-_i=0$. Then $M^{\boldsymbol{\alpha}}(\boldsymbol{\nu},(L,\nabla_L))^0$ is isomorphic to the total space of the cotangent bundle $T^*V_0$ and the map $\Bun \colon M^{\boldsymbol{\alpha}}(\boldsymbol{\nu},(L,\nabla_L))^0 \rightarrow V_0$ corresponds to the natural projection $T^*V_0 \rightarrow V_0$. Moreover, the section $\nabla_0 \colon V_0 \rightarrow M^{\boldsymbol{\alpha}}(\boldsymbol{\nu},(L,\nabla_L))^0$ corresponding to the zero section $V_0 \rightarrow T^*V_0$ is given by those reducible connections preserving the destabilizing subbundle $\mathcal{O}_C$.
	\end{proposition}
	\subsection{Another proof of Theorem \ref{MT}}
	We will show $\App \times \Bun$ is a birational map in another way. First, we show the existence of a parabolic connection over a given parabolic bundle. The following lemma is an analogy of Lemma 2.5 in \cite{FL}.
	\begin{lemma}\label{lemexconne}
		Suppose that $\sum_{i=1}^nw_i < 1$. Then for each $(E, l_*) \in P^{\boldsymbol{\alpha}}(L)$, there is a $\boldsymbol{\nu}$-parabolic connection $(E,\nabla, l_*)$ such that $\tr \nabla \simeq \nabla_L$.
	\end{lemma}
	\begin{proof}
		Let $\{U_i\}_i$ be an open covering of $C$ and $\nabla'_i$ be a logarithmic connection on $U_i$ satisfying $(\textrm{res}_{t_k}(\nabla'_i)-\nu^+_k\textrm{id})(l^{(k)}_1)=0, (\textrm{res}_{t_k}(\nabla'_i)-\nu^-_k\textrm{id})(E|_{t_i}) \subset l^{(k)}_1$
		at each $t_k \in U_i$ and $\tr \nabla'_i=\nabla_L|_{U_i}$. We define sheaves $\mathcal{E}_0$ and $\mathcal{E}_1$ on $C$ by 
		\begin{align*}
			\mathcal{E}^0
			&:=\{s \in \mathcal{E}nd(E) \mid \text{$\tr (s)=0$ and $s_{t_i}(l^{(i)}_1) \subset l^{(i)}_1$ for any $i$}\}, \\
			\mathcal{E}^1
			&:=\{s \in \mathcal{E}nd(E) \otimes \Omega_C^1(D) \mid \text{$\tr (s)=0$ and $\textrm{res}_{t_i}(s)(l^{(i)}_j)\subset l^{(i)}_{j+1}$ for any $i, j$}\}.
		\end{align*}
		Then the isomorphism $\mathcal{E}^1 \simeq (\mathcal{E}^0)^\vee \otimes \Omega_C^1$ holds.
		Differences $\nabla'_i-\nabla'_j$ define the cocycle
		\[
		(\nabla'_i-\nabla'_j)_{i,j} \in H^1(C,\mathcal{E}^1).
		\]
		By Serre duality and the simplicity of $E$, we obtain
		\[
		H^1(C,\mathcal{E}^1) \simeq H^0(C, \mathcal{E}^0)^* = \{0\}.
		\]
		Hence, there exists $\Phi_i \in \mathcal{E}^1(U_i)$ for each $i$ such that $\nabla'_i-\nabla'_j=\Phi_i-\Phi_j$. Set $\nabla_i=\nabla'_i-\Phi_i$. 
		Then $(\nabla_i)_i$ defines a $\boldsymbol{\nu}$-parabolic connection $\nabla$ over $(E, l_*)$ satisfying $\tr \nabla \simeq \nabla_L$.
	\end{proof}
	For a quasi-parabolic bundle $(E, l_*) \in V_0$, let us fix a $\boldsymbol{\nu}$-parabolic connection $(E,\nabla, l_*) \in \Bun^{-1}((E, l_*))$.  Let $(E,\nabla', l_*) \in \Bun^{-1}((E, l_*))$ be another $\boldsymbol{\nu}$-parabolic connection. Then 
	$\nabla'-\nabla$ is a global section of $\mathcal{E}^1$ which is the sheaf defined in the proof of Lemma \ref{lemexconne}. Therefore, we have the isomorphism $\Bun^{-1}((E, l_*)) \simeq \nabla+H^0(C,\mathcal{E}^1)$. 
	
	For a section $\Theta \in H^0(\mathcal{E}nd(E) \otimes \Omega_C^1(D))$, we define the section $\varphi_\Theta \in H^0(C, L \otimes \Omega_C^1(D))$ by the composite
	\[
	\mathcal{O}_C  \hookrightarrow E \xrightarrow{\Theta} E\otimes \Omega_C^1(D) \rightarrow E/\mathcal{O}_C \otimes \Omega_C^1(D) \simeq L \otimes \Omega_C^1(D)
	\]
	and define the map
	\[
	\varphi \colon H^0(C,\mathcal{E}nd(E) \otimes \Omega_C^1(D)) \longrightarrow H^0(C, L \otimes \Omega^1(D))  
	\]
	by $\varphi(\Theta)=\varphi_\Theta$. It is clearly linear. Let us define the sheaf $\mathcal{F}^1$ by
	\[
	\mathcal{F}^1
	=\{s \in \mathcal{E}nd(E) \otimes \Omega_C^1(D) \mid \text{$\textrm{res}_{t_i}(s)(l^{(i)}_j)\subset l^{(i)}_{j+1}$ for all $i, j$}\}.
	\]
	Assume that $\Theta \in H^0(C,\mathcal{F}^1)$ satisfies $\varphi_\Theta=0$, that is, $\Theta(\mathcal{O}_C) \subset \mathcal{O}_C \otimes \Omega_C^1(D)$. By definitions of $V_0$ and $\mathcal{E}^1$, we obtain 
	$\textrm{res}_{t_i}(\Theta)(\mathcal{O}_C|_{t_i}) \subset \mathcal{O}_C|_{t_i} \cap l^{(i)}_1=\{0\}$ for any $i$. Hence, we have $\Theta(\mathcal{O}_C) \subset \mathcal{O}_C \otimes \Omega_C$, that is, $\Theta|_{\mathcal{O}_C}$ is a global section of $\Omega_C^1$.
	\begin{lemma}\label{lemtheta}
		The linear map $H^0(C,\mathcal{F}^1)\cap\Ker \varphi \rightarrow H^0(C,\Omega_C^1),\ \Theta \mapsto \Theta|_{\mathcal{O}_C}$ is an isomorphism.
	\end{lemma}
	\begin{proof}
		For $\mu \in H^0(C,\Omega_C^1)$, we define $\Theta=\id_E\otimes \mu$.  Then we have $\Theta \in H^0(C,\mathcal{F}^1)\cap\Ker \varphi$ and $\Theta|_{\mathcal{O}_C}=\mu$. The linear map is hence surjective. We show that the map is injective. If $\Theta \in H^0(C,\mathcal{F}^1)\cap\Ker \varphi$ satisfies $\Theta|_{\mathcal{O}_C}=0$, then $\Theta$ induces the homomorphism $\hat{\Theta} \colon L \simeq E/\mathcal{O}_C \rightarrow E \otimes \Omega_C^1(D)$. 
		$\textrm{res}_{t_i}(\Theta)=0$ implies $\textrm{res}_{t_i}(\hat{\Theta})=0$, so we obtain $\hat{\Theta}(L) \subset E \otimes \Omega_C^1$.
		Since $\rank E=2$, we have isomorphisms $E^\vee \simeq E \otimes (\det E)^{-1} \simeq E \otimes L^{-1}$. By this isomorphism and Serre duality, 
		\[
		\Hom(L,E \otimes \Omega_C^1) \simeq H^0(C, L^{-1} \otimes E \otimes \Omega_C^1)\simeq  H^0(C,E^\vee \otimes \Omega_C^1) \simeq H^1(C,E)^*=\{0\}
		\]
		Hence we obtain $\hat{\Theta}=0$ and this implies $\Theta=0$. 
	\end{proof}
	\begin{proof}(Another proof the second assertion of Theorem \ref{MT})
		
		We show that for each $(E, l_*) \in V_0$, the morphism
		\[
		\App \colon \Bun^{-1}((E, l_*)) \longrightarrow \mathbb{P} H^0(C, L\otimes \Omega_C^1(D))
		\]
		is injective.
		
		Let us fix a $\boldsymbol{\nu}$-parabolic connection $(E,\nabla, l_*) \in \Bun^{-1}((E, l_*))$. If there exists $\Theta \in H^0(C,\mathcal{E}^1)$ such that 
		$\varphi_\nabla=\varphi_\Theta$, then $\nabla-\Theta$ is a $\boldsymbol{\nu}$-parabolic connection and $\varphi_{\nabla-\Theta}=0$. It is a contradiction. Thus, we have 
		\[
		\{\varphi_\Theta \mid \Theta \in H^0(C,\mathcal{E}^1)\} \cap \mathbb{C}\varphi_\nabla =\{0\}.
		\]
		Hence, we only need to show that the linear map $\varphi \colon H^0(C,\mathcal{E}^1) \rightarrow H^0(C, L \otimes \Omega^1(D))$ is injective. Suppose that a section $\Theta \in H^0(C,\mathcal{E}^1)$ satisfies $\varphi_\Theta=0$. By the proof of Lemma \ref{lemtheta},
		there is a section $\mu \in H^0(C,\Omega_C^1)$ such that $\Theta = \id_E \otimes \mu$. Since $\tr\, \Theta =0$, we get $\mu=0$ and this means $\Theta=0$.
	\end{proof}
	\begin{remark}\label{rationality}
	 Theorem \ref{MT} provides a proof of the rationality of the moduli space. On the other hand, the rationality of moduli space $M^{\boldsymbol{\alpha}}(\boldsymbol{\nu},(L,\nabla_L))$ of $\boldsymbol{\alpha}$-stable $\boldsymbol{\nu}$-parabolic connections with the determinant $(L,\nabla_L)$ over a curve $C$ with genus $g\geq 2$ for any rank $r$ can be proven without using $\App\times \Bun$ as follows. (This was suggested by the anonymous referee.) Let $P^{\boldsymbol{\alpha}}(r, L)$ be the moduli space of $\boldsymbol{\alpha}$-stable parabolic bundles of rank $r$ over $(C,\bm{t})$ with the determinant $L$ and $\eta$ be the scheme theoretic generic point of $P^{\boldsymbol{\alpha}}(r, L)$, which corresponds to a quasi-parabolic bundle $(E_\eta,l_{\eta,*})$ on $C_\eta:=C\times \textrm{Spec}\, k(\eta)$. Since for any $\boldsymbol{\alpha}$-stable parabolic bundles a traceless endomorphism compatible with the parabolic structure is only zero, we can construct a parabolic $\boldsymbol{\nu}$-connection on $(E_\eta,l_{\eta,*})$ with the determinant $(L,\nabla_L)_\eta$ in the same manner as the proof of Lemma \ref{lemexconne}. This means that there is a Zariski-open subset $U\subset P^{\boldsymbol{\alpha}}(r, L)$ and a flat family $(\tilde{E},\tilde{\nabla}, \tilde{l}_*)$ of $\boldsymbol{\nu}$-parabolic connections on $C\times U$ over $U$ such that $(\tilde{E},\tilde{l}_*)$ coincides with the restriction of the universal family on $C\times P^{\boldsymbol{\alpha}}(r, L)$. So we get a morphism
	 \[
	 T^*U\ni \Phi \longmapsto \tilde{\nabla}+\Phi \in M^{\boldsymbol{\alpha}}(\boldsymbol{\nu},(L,\nabla_L)),
	 \]
	 which is an isomorphism onto the Zariski-open subset of $M^{\boldsymbol{\alpha}}(\boldsymbol{\nu},(L,\nabla_L))$ consisting of parabolic connections whose underlying parabolic bundles belong to $U$. Since $P^{\boldsymbol{\alpha}}(r, L)$ is rational by \cite{BY}, $M^{\boldsymbol{\alpha}}(\boldsymbol{\nu},(L,\nabla_L))$ is also rational.
	\end{remark}
	\begin{subsection}{Lagrangian fibrations}
		Recall the canonical symplectic structure on $M^{\boldsymbol{\alpha}}(\boldsymbol{\nu},(L,\nabla_L))$ (see section 6 in \cite{IIS1} and section 7 in \cite{In} for more detail). Take a point $x=(E,\nabla, l_*)\in M^{\boldsymbol{\alpha}}(\boldsymbol{\nu},(L,\nabla_L))$. Let $\mathcal{E}^\bullet$ be the complex of sheaves defined by 
		\[
		\mathcal{E}^0\longrightarrow\mathcal{E}^1, \; s\longmapsto \nabla\circ s-s\circ\nabla, 
		\]
		where $\mathcal{E}^0$ and $\mathcal{E}^1$ are sheaves defined in Lemma \ref{lemexconne}. Then there exists the canonical isomorphism between the tangent space $T_xM^{\boldsymbol{\alpha}}(\boldsymbol{\nu},(L,\nabla_L))$ and the hypercohomology group $\mathbf{H}^1(\mathcal{E}^\bullet)$. Take an open covering $\{U_i\}_i$ of $C$. In \v{C}ech cohomology an element of $\mathbf{H}^1(\mathcal{E}^\bullet)$ can be written by the form $\{(B_{ij}), (\Phi_i)\}$, where $(B_{ij})_{i,j} \in C^1(\mathcal{E}^0)$, $(\Phi_i)_i \in C^0(\mathcal{E}^1)$ and $(\nabla B_{ij}-B_{ij}\nabla)_{i,j}=(\Phi_j-\Phi_i)_{i,j}$ in $C^1(\mathcal{E}^1)$. The canonical symplectic form $\Omega$ on $M^{\boldsymbol{\alpha}}(\boldsymbol{\nu},(L,\nabla_L))$ is defined by
		\begin{align*}
			\Omega_x\colon \mathbf{H}^1(\mathcal{E}^\bullet)\otimes \mathbf{H}^1(\mathcal{E}^\bullet)&\longrightarrow \mathbf{H}^2(\mathcal{O}_C\overset{d}{\rightarrow}\Omega_C^1)\cong \mathbb{C}\\
			(\{(B_{ij}), (\Phi_i)\},\{(B'_{ij}), (\Phi'_i)\}) &\longmapsto (\{\tr(B_{ij}\circ B'_{jk})\}, -\{(\tr (B_{ij}\circ\Phi'_j)-\tr (\Phi_i\circ B'_{ij}))\})
		\end{align*}
		at each $x$. We can see that the homomorphisms $H^0(C,\mathcal{E}^1)\rightarrow \mathbf{H}^1(C,\mathcal{E}^\bullet)$ and $\mathbf{H}^1(C,\mathcal{E}^\bullet)\rightarrow H^1(C,\mathcal{E}^0)$ defined by $(\Phi_i)_i \mapsto \{0, (\Phi_i)_i\}$ and $\{(B_{ij})_{i,j}, (\Phi_i)_i\} \mapsto (B_{ij})_{i,j}$, respectively, give an exact sequence 
		\[
		H^0(C,\mathcal{E}^0)\longrightarrow H^0(C,\mathcal{E}^1)\longrightarrow \mathbf{H}^1(C,\mathcal{E}^\bullet)\longrightarrow H^1(C,\mathcal{E}^0)\longrightarrow H^1(C,\mathcal{E}^1).
		\]
		When $(E,\nabla, l_*)\in M^{\boldsymbol{\alpha}}(\boldsymbol{\nu},(L,\nabla_L))^0$, we have  $H^1(C,\mathcal{E}^1) \simeq H^0(C, \mathcal{E}^0)^* = \{0\}$. We note that each element in $H^1(C,\mathcal{E}^0)$ gives a deformation of $(E, l_*)$. 
		\begin{proposition}
			$\App\colon M^{\boldsymbol{\alpha}}(\boldsymbol{\nu},(L,\nabla_L))^0\rightarrow |L\otimes \Omega_C^1(D)|$ and $\Bun\colon M^{\boldsymbol{\alpha}}(\boldsymbol{\nu},(L,\nabla_L))^0\rightarrow V_0$ are Lagrangian fibrations.
		\end{proposition}
		\begin{proof}
			Take a point $x=(E,\nabla, l_*)\in M^{\boldsymbol{\alpha}}(\boldsymbol{\nu},(L,\nabla_L))^0$ and put $[\gamma]=\App (x)$ and $[b]=\Bun(x)$, where $\gamma=(\gamma_i)_i\in H^0(C, L\otimes \Omega_C^1(D))$ and $b=(b_{ij})_{i,j}\in H^1(C, L^{-1}(-D))$ are nonzero elements. Then a transition matrix $M_{ij}$ of $E$ and a connection matrix $A_i$ of $\nabla$ have the form
			\[
			M_{ij}=
			\begin{pmatrix}
				1&b_{ij}\\
				0&c_{ij}
			\end{pmatrix}, \;
			A_{i}=
			\begin{pmatrix}
				\alpha_i&\beta_i\\
				\gamma_i&\delta_i
			\end{pmatrix}, 
			\]
			respectively. The natural homomorphism 
			\[
			T_x\App^{-1}([\gamma])\oplus T_x\Bun^{-1}([b])\longrightarrow T_xM^{\boldsymbol{\alpha}}(\boldsymbol{\nu},(L,\nabla_L))\cong \mathbf{H}^1(C,\mathcal{E}^\bullet)
			\]
			is an isomorphism. Since any element in $T_x\Bun^{-1}([b])$ does not deform $(E, l_*)$, we have $T_x\Bun^{-1}([b])\subset H^0(C,\mathcal{E}^1)$. So $\Omega|_{\Bun^{-1}([b])}=0$ and $T_x\App^{-1}([\gamma])\rightarrow H^1(C,\mathcal{E}^0)$ is an isomorphism. Take $\{(B_{ij})_{ij}, (\Phi_i)_i\} \in \mathbf{H}^1(C,\mathcal{E}^\bullet)$. Since the homomorphism
			\[
			T_{[b]}\mathbb{P}H^1(C, L^{-1}(-D))\cong H^1(C, L^{-1}(-D))/[b]\rightarrow H^1(C,\mathcal{E}^0)\cong T_{(E, l_*)}P^{\boldsymbol{\alpha}}(L), \; (g_{ij})_{i,j}\mapsto
			\left(
			\begin{pmatrix}
				0&g_{ij}\\
				0&0
			\end{pmatrix}\right)_{i,j}
			\]
			is an isomorphism, 
			$B_{ij}$ and $\Phi_i$ can be written by the form 
			\[
			B_{ij}=
			\begin{pmatrix}
				0&g_{ij}\\
				0&0
			\end{pmatrix},\;
			\Phi_i=
			\begin{pmatrix}
				\zeta_i&\eta_i\\
				\theta_i&-\zeta_i
			\end{pmatrix}, 
			\]
		where $\zeta_i, \eta_i, \in \Omega_C^1(U_i)$ and $\theta_i \in \Omega_C^1(D)(U_i)$. We note that $(b_{ij}\gamma_j)_{i,j}$ is a nonzero cocycle in $H^1(C,\Omega_{C}^1)$ (see Step 2 in the proof of Theorem \ref{MT}). So we have $H^1(C, L^{-1}(-D))=[b]\oplus \Ker \langle\gamma,\;\rangle$, where $\langle\; , \; \rangle$ is the natural pairing
		\[
		\langle\; , \; \rangle\colon H^0(C, L\otimes \Omega_C^1(D))\times H^1(C, L^{-1}(-D))\longrightarrow H^1(C,\Omega_C^1).
		\]
		Since $b=0$ in $H^1(C,\mathcal{E}^0)$, the composite 
		\[
		\Ker \langle\gamma,\;\rangle \rightarrow H^1(C, L^{-1}(-D))\rightarrow H^1(C,\mathcal{E}^0)
		\]
		becomes an isomorphism. So we may assume that 
		$(g_{ij})_{i,j} \in \Ker \langle\gamma, \;\rangle$. The condition  $\nabla B_{ij}-B_{ij}\nabla=dB_{ij}+A_iB_{ij}-B_{ij}A_j=M_{ij}\Phi_j-\Phi_iM_{ij}$ is equivalent to
		\begin{equation*}
			\left\{
			\begin{array}{ll}
				-g_{ij}\gamma_j=\zeta_j-\zeta_i+b_{ij}\theta_j \\
				dg_{ij}+\alpha_ig_{ij}-g_{ij}\delta_j=\eta_j-\eta_ic_{ij}-b_{ij}(\zeta_i+\zeta_j)\\
				c_{ij}\theta_j-\theta_i=0.
			\end{array}
			\right.
		\end{equation*}
		So $\theta=(\theta_i)_i$ defines a global section of $L\otimes \Omega_{C}^1(D)$ and $(b_{ij}\theta_j)_{i,j}$ is zero in $H^1(C, \Omega_C^1)$. Assume that $\{(B_{ij})_{ij}, (\Phi_i)_i\}\in  T_x\App^{-1}([\gamma])$. Then $\theta$ is an element of $[\gamma]$, and so $\theta$ must be zero. 
		Hence we have 
		\[
		\Omega_x(\{(B_{ij})_{ij}, (\Phi_i)_i\}, \{(B'_{ij})_{ij}, (\Phi'_i)_i\} )=0
		\]
		for any $\{(B_{ij})_{ij}, (\Phi_i)_i\}, \{(B'_{ij})_{ij}, (\Phi'_i)_i\}\in T_x\App^{-1}([\gamma])$, which means that $\Omega|_{\App^{-1}([\gamma])}=0$.
		\end{proof}
	\end{subsection}
	\section*{Acknowledgments}
	The author would like to thank Arata Komyo, Masa-Hiko Saito, and Kota Yoshioka for useful comments and valuable discussions. He also thanks the anonymous referee for improving this paper and suggesting a proof of the rationality of the moduli space of parabolic connections for any rank. He is supported by Japan Society for the Promotion of Science KAKENHI Grant Numbers 22J10695.
	
\end{document}